\documentclass{article}

\usepackage{makeidx}
\usepackage{latexsym}
\usepackage{amsfonts}
\usepackage{amssymb}
\usepackage{amsmath}
\usepackage{amstext}
\usepackage{amsthm}
\usepackage{mathrsfs}

\usepackage{rotating}

\newtheorem*{thmA}{Theorem A}

\newtheorem{theorem}{Theorem}

\newtheorem{lemma}[theorem]{Lemma}

\newtheorem{corollary}[theorem]{Corollary}

\newtheorem{proposition}[theorem]{Proposition}

\begin{document}

\title{Symmetric majority rules\footnote{  We are grateful to two
anonymous referees for providing useful suggestions allowing to improve the readability of the paper. In particular, one of
the referees suggested an interesting link between the minimal majority principle and the method of simple majority decision (see Proposition \ref{simple} and the final remark in Section \ref{3-3}).
}}
\author{\textbf{Daniela Bubboloni} \\
{\small {Dipartimento di Scienze per l'Economia e  l'Impresa} }\\
\vspace{-6mm}\\
{\small {Universit\`{a} degli Studi di Firenze} }\\
\vspace{-6mm}\\
{\small {via delle Pandette 9, 50127, Firenze}}\\
\vspace{-6mm}\\
{\small {e-mail: daniela.bubboloni@unifi.it}} \and \textbf{Michele Gori}
 \\
{\small {Dipartimento di Scienze per l'Economia e  l'Impresa} }\\
\vspace{-6mm}\\
{\small {Universit\`{a} degli Studi di Firenze} }\\
\vspace{-6mm}\\
{\small {via delle Pandette 9, 50127, Firenze}}\\
\vspace{-6mm}\\
{\small {e-mail: michele.gori@unifi.it}}}

\maketitle

\begin{abstract}
\noindent In the standard arrovian framework and under the assumption that  individual preferences and social outcomes are linear orders on the set of alternatives, we study the rules which satisfy suitable symmetries and obey the majority principle. In particular, supposing that individuals and alternatives are exogenously partitioned into subcommittees and subclasses, we provide necessary and sufficient conditions for the existence of  reversal symmetric majority rules that are anonymous and neutral with respect to the considered partitions.
We also determine a general method for constructing and counting those rules and we explicitly apply it to some simple cases.
\end{abstract}

\vspace{4mm}


\vspace{2mm}


\section{Introduction}

Committees are often required to provide a strict ranking of a given family of alternatives.
There are many procedures that members of a committee can conceive to aggregate their preferences on alternatives into a strict ranking of these alternatives. Among them the ones satisfying the principles of anonymity and neutrality are usually preferred. The principle of anonymity is the requirement that the identities of individuals are irrelevant to determine the social outcome. The principle of neutrality is instead the requirement that alternatives are equally treated. Unfortunately, despite their appeal, these principles can be both satisfied by an aggregation procedure only in very special circumstances.

Consider a committee having $h\ge 2$ members whose purpose is to strict rank $n\ge 2$ alternatives, and assume that individual and social preferences are strict rankings on the set of alternatives. A preference profile is a list of $h$ strict rankings each of them associated with the name of a specific individual and representing her preferences. Any function from the set of preference profiles to the set of social preferences is called a rule and represents a particular decision process which determines a social ranking of alternatives, whatever individual preferences the committee members express. In such a framework, Bubboloni and Gori (2014, Theorem 5) prove that it is possible to design anonymous and neutral rules if and only if
\begin{equation}\label{moulin-condition}
\gcd(h,n!)=1.
\end{equation}
Condition \eqref{moulin-condition}, first introduced by Moulin (1983, Theorem 1, p.25) as a necessary and sufficient condition for the existence of anonymous, neutral and efficient social choice functions, is a very strong arithmetical condition rarely satisfied in concrete situations. When it fails we can only try to design rules satisfying weaker versions of the principles of anonymity and neutrality.

A possible way to weaken anonymity is to divide individuals into subcommittees and require that, within each subcommittee, individuals equally influence the final collective decision, while individuals belonging to different subcommittees may have a different decision power. Analogously,
we can weaken neutrality thinking alternatives to be divided into subclasses and assuming that within each subclass alternatives are equally treated,  while alternatives belonging to different subclasses may have a different treatment. These versions of anonymity and neutrality are certainly natural and actually used in many practical collective decision processes.
That happens, for instance, when a committee has a president working as a tie-breaker or when a committee evaluates job candidates discriminating on their gender. Indeed, in the former example committee members can be thought to be divided in two subcommittees (the president in the first, all the others in the second) with anonymous individuals within each of them; in the latter example alternatives can be thought to be divided in two subclasses (the women in the first, the men in the second) such that no alternative has an exogenous advantage with respect to the other alternatives in the same subclass.

The formalization of those new concepts is natural. In fact, given a partition of individuals into subcommittees, we say that a rule is anonymous with respect to those subcommittees if it has the same value over any pair of preference profiles such that we can get one from the other by figuring to permute the names of individuals belonging to the same subcommittee. Given instead a partition of alternatives into subclasses, we say that a rule is neutral with respect to those subclasses if, for every pair of preference profiles such that we can get one from the other by figuring to permute the names of alternatives belonging to the same subclass, the social preferences associated with them coincide up to the considered permutation.
Of course, requiring that a rule is anonymous (neutral) is equivalent both to require that it is anonymous (neutral) with respect to the partition whose unique element is the whole set of individuals (alternatives), and to require that it is anonymous (neutral) with respect to any partition of individuals (alternatives).

Certainly, beyond anonymity and neutrality, social choice theorists identify further principles that rules should meet. The majority and the reversal symmetry principles are some of them. Roughly speaking, the majority principle requires that if a large enough amount of people prefer an alternative to another one, then the former alternative must be socially preferred to the latter one. In the literature we can find several ways to interpret that principle, such as relative majority, absolute majority, qualified majority and so on; here we focus on the minimal majority principle introduced by Bubboloni and Gori (2014).  Given an integer $\nu$, called majority threshold, not exceeding the number of members in the committee but exceeding half of it and a preference profile, we say that a social preference is consistent with the $\nu$-majority principle applied to the considered preference profile if the fact that an
alternative is preferred to another one by at least $\nu$ individuals implies that the alternative is socially ranked over the other one.
A rule is said to be a minimal majority rule if it associates with every preference profile a social preference which is consistent with the $\nu$-majority principle applied to that preference profile for all majority threshold $\nu$ not generating Condorcet-cycles. The principle of reversal symmetry states instead that if everybody in the society completely changes her mind about her own ranking of alternatives, then a complete change in the social outcome occurs. It can be formally described by recalling first that, given a preference, its reversal is the preference obtained making the best alternative the worst, the second best alternative  the second worst, and so on. A rule is then reversal symmetric if, for any pair of preference profiles such that one is obtained by the other reversing each individual preference, the social outcomes associated with them are one the reversal of the other.

In the present paper we analyse the rules that satisfy  anonymity with respect to subcommittees and neutrality with respect to subclasses, and also obey the principles of minimal majority and reversal symmetry. At the best of our knowledge, conditions assuring the existence of those rules  are not known. Some contributions related to different notions of anonymity and neutrality and their link with the majority principle are instead present in the literature. Under the assumption that there are two alternatives and assuming the possibility of indifference in individual and social preferences, Perry and Powers (2008) calculate the number of rules that satisfy anonymity and neutrality and the number of rules satisfying
a restrictive version of anonymity (that is, every individual but one is anonymous) and neutrality. In the same framework, Powers (2010) further shows that an aggregation rule satisfies that restrictive version of anonymity, neutrality and Maskin monotonicity if and only if it is close to an absolute qualified majority rule.
Quesada (2013) identifies instead seven axioms (among which are weak versions of anonymity and neutrality) characterizing the rules that are either the relative majority rule or the relative majority rule where a given individual, the chairman, can break the ties.
In the framework of social choice functions, Campbell and Kelly (2011, 2013) show that the relative majority is implied both by a suitable weak version of anonymity, neutrality and monotonicity, as well as by what they called limited neutrality, anonymity and monotonicity. Moreover, in the general case for the number of alternatives, some observations about different levels of anonymity and neutrality  can be found in the paper by Kelly (1991), who uses the language of permutations groups to discuss some open problems.

Here we follow the algebraic approach developed in Bubboloni and Gori (2014) to carry on our analysis, and we also adhere to the framework and notation used there. In that paper, which we refer to for further references on anonymity, neutrality and majority principles, the authors show how the notion of action of a group on a set can naturally and fruitfully be used to study problems concerning anonymity and neutrality. Indeed, among other things, they prove that condition \eqref{moulin-condition} is necessary and sufficient for the existence of anonymous and neutral minimal majority rule.\footnote{See Theorem 14 in Bubboloni and Gori (2014).} In this paper we adapt that algebraic reasoning in order to treat  anonymity with respect to subcommittees and neutrality with respect to subclasses, together with reversal symmetry and minimal majority. We obtain, as our main result, the following theorem.\footnote{Theorem A is a rephrasing of Theorem \ref{comm}.}
\begin{thmA}
Assume that individuals are partitioned into $s\ge 1$ subcommittees with number of members $b_1,\ldots,b_s$, and that
alternatives are partitioned into $t\ge 1$ subclasses with number of alternatives $c_1,\ldots,c_t$. Then:
\begin{itemize}
	\item[i)] there exists a minimal majority rule which is anonymous with respect to the considered subcommittees and neutral with respect to the considered subclasses if and only if
\begin{equation}\label{nostra0}
\gcd\big(\gcd (b_1,\ldots,b_s), \mathrm{lcm}(c_1!,\ldots,c_t!)\big)=1;
\end{equation}
	\item[ii)] there exists a reversal symmetric minimal majority rule which is anonymous with respect to the considered subcommittees and neutral with respect to the considered subclasses if and only if
\begin{equation}\label{nostra}
\gcd\big(\gcd (b_1,\ldots,b_s), \mathrm{lcm}(2,c_1!,\ldots,c_t!)\big)=1.
\end{equation}
\end{itemize}
\end{thmA}

%

Note that
\eqref{nostra} obviously implies \eqref{nostra0}, and that  \eqref{nostra0} and \eqref{nostra} are equivalent if one among the $c_1,\dots, c_t$ is greater than $1$. 
Since \eqref{moulin-condition} implies \eqref{nostra},
Theorem A generalizes many results proved in Bubboloni and Gori (2014).
In particular, condition \eqref{moulin-condition} is sufficient not only to get anonymous and neutral minimal majority rules, but also to get rules having the further property of being reversal symmetric (Corollary \ref{families}). Yet, Theorem A goes much beyond that. Indeed, it shows that if  \eqref{moulin-condition} does not hold true but
the specific purpose of the collective choice naturally allows partitions of the individuals and the alternatives into subcommittees and subclasses satisfying  \eqref{nostra}, then it is possible to design a minimal majority rule which is anonymous and neutral with respect to those partitions as well as reversal symmetric. That happens, as a special but remarkable case, when
all the members of a committee but one are anonymous (for instance, the committee has a president),
 independently of the partition of alternatives in subclasses (Corollary \ref{president}).

We finally want to emphasize that our algebraic approach actually allows us to define a very general and wide-ranging notion of symmetry for rules (Section \ref{sym-rul}), which includes anonymity with respect to subcommittees, neutrality with respect to subclasses and reversal symmetry as particular instances. That notion of symmetry provides a fruitful unified framework with a double advantage: arguments and proofs becomes simpler and  more direct; the results reach a very satisfying level of generality (Theorems \ref{main} and \ref{mainmain}). Moreover, as in Bubboloni and Gori (2014), the algebraic machinery provides a method to potentially build all the rules described in Theorem A. In Section \ref{example} we briefly discuss some examples that explain how the theoretical results can be explicitly applied in some simple cases.

\section{Definitions and notation}

\subsection{Linear orders and permutations}

Let $X$ be a nonempty finite set. We denote by $\mathcal{R}(X)$ the set of relations on $X$. Given $R\in \mathcal{R}(X)$ and $x,y\in X$, we sometimes write $x\ge_{R}y$ instead of $(x,y)\in R$, as well as $x>_{R}y$ instead of $(x,y)\in R$ and $(y,x)\not \in R$. If $R\in  \mathcal{R}(X)$ is antisymmetric, then $x>_{R}y$ is equivalent to $x\ge _{R}y$ and $x\neq y$. 
A relation on $X$ is called a linear order on $X$  if it is complete, transitive and antisymmetric. The set of linear orders on $X$ is denoted by $\mathcal{L}(X)$. If $R_1,R_2\in  \mathcal{L}(X)$, then $R_1=R_2$ if and only if, for every $x,y\in X$,  $x>_{R_1}y$ implies  $x>_{R_2}y$.

We denote by $\mathrm{Sym}(X)$ the group of the bijective functions from $X$ to itself, with product given by the right-to-left composition, that is, if $f_1,f_2\in\mathrm{Sym}(X)$, then $f_1f_2\in \mathrm{Sym}(X)$ is the function such that, for every $x\in X$, $f_1f_2(x)=f_1(f_2(x))$. The neutral element of $\mathrm{Sym}(X)$ is given by the identity function $id$.
$\mathrm{Sym}(X)$ is called the symmetric group\footnote{The notation and results of group theory about  permutation groups and actions, not explicitly discussed in the paper, are standard (see, for instance, Wielandt (1964) and Rose (1978)).} on $X$ and its elements  permutations on $X$. Given $\sigma\in \mathrm{Sym}(X)$, we denote its order by $|\sigma|$.
For every $k\in\mathbb{N}$, the group $\mathrm{Sym}(\{1,\dots,k\})$ is simply denoted by $S_{k}$.

\subsection{Preference relations}

\noindent From now on, let $n\in \mathbb{N}$ with $n\ge 2$ be fixed, and let $N=\{1,\ldots,n\}$ be the set of alternatives.  A {\it preference relation} on $N$ is an element of $\mathcal{L}(N)$.
Throughout the section, let $q\in\mathcal{L}(N)$ be fixed. For every $x,y\in N$, we say that $x$ is  preferred to $y$ according to $q$ if $x>_{q} y$. For every $\psi \in S_n$, we define $\psi q$ as the element of $\mathcal{L}(N)$ such that, for every $x,y\in N$, $(x,y)\in \psi q$ if and only if $(\psi^{-1}(x),\psi^{-1}(y))\in q$. 
Consider the {\it order reversing permutation} in $S_n$, that is, the permutation $\rho_0\in S_n$ defined, for every $r\in \{1,\ldots,n\}$, as $\rho_0(r)=n-r+1$. Note that $|\rho_0|=2$.
We define $q\rho_0\in\mathcal{L}(N)$ as the element in $\mathcal{L}(N)$ such that, for every $x,y\in N$,
$(x,y)\in q\rho_0$ if and only if $(y,x)\in q$. We also define $q\; id=q$, where $id\in S_n$.
By definition, for every $x,y\in N$ and $\psi\in S_n$, we have that
\begin{equation}\label{psiR}
x>_{q}y \quad \mbox{if and only if} \quad \psi(x)>_{\psi q}\psi(y),
\end{equation}
and
\begin{equation}\label{ro}
x>_{q} y \quad \mbox{if and only if}\quad  y>_{q\rho_0} x.
\end{equation}
Consider now the set of vectors with $n$ distinct components in $N$ given by
\[
\mathcal{V}(N)=\left\{(x_r)_{r=1}^n\in N^n: x_{r_1}=x_{r_2}\Rightarrow r_1=r_2\right\},
\]
and think each vector $(x_r)_{r=1}^n\in \mathcal{V}(N)$ as a column vector, that is,
 \[
(x_r)_{r=1}^n=\begin{bmatrix}
x_1\\ \vdots \\ x_n
\end{bmatrix}=[x_1,\ldots,x_n]^T.
\]
The function $f_1:\mathcal{V}(N)\to \mathcal{L}(N)$ associating with  $(x_r)_{r=1}^n\in\mathcal{V}(N)$ the preference relation
\[
\{(x_{r_1},x_{r_2})\in N\times N: r_1,r_2\in \{1,\ldots,n\}, r_1\le r_2\},
\]
and the function
$f_2:S_n\to \mathcal{L}(N)$ associating with $\sigma\in S_n$  the preference relation
\[
\{(\sigma(r_1),\sigma(r_2))\in N\times N: r_1,r_2\in \{1,\ldots,n\}, r_1\le r_2\}
\]
are bijective, so that, in particular,   $|S_n|=|\mathcal{V}(N)|=|\mathcal{L}(N)|=n!$. We say that $x\in N$ has rank $r\in \{1,\ldots,n\}$ in $q$ if $x$ is the $r$-th component of $f_1^{-1}(q)$ or, equivalently, if $x$ is the image of $r$ through $f_2^{-1}(q)$.
Note now that, for every $\psi\in S_n$ and $\rho\in\{id,\rho_0\}$,  if $f_1^{-1}(q)=[x_1,\ldots,x_n]^T$, then
\[
f_1^{-1}(\psi q)=[\psi(x_1),\ldots,\psi(x_n)]^T,\quad \mbox{and}\quad f_1^{-1}(q\rho)=[x_{\rho(1)},\ldots,x_{\rho(n)}]^T;
\]
if $f_2^{-1}(q)=\sigma$, then 
\[
f_2^{-1}(\psi q)=\psi\sigma,\quad \mbox{and}\quad f_2^{-1}(q\rho)=\sigma \rho. 
\]
Thus, by the functions $f_1$ and $f_2$ we are allowed to identify the preference relation $q$ both with the vector $f_1^{-1}(q)$ and with the permutation  $f_2^{-1}(q)$, and to naturally interpret the products $\psi q$ and $q\rho$ in $\mathcal{V}(N)$ and in $S_n$.
For instance, if $n=4$ and 
\[
q=\left\{(4,2),(2,1),(1,3),(4,1),(4,3),(2,3),(4,4),(2,2), (1,1),(3,3)\right\}\in\mathcal{L}(\{1,2,3,4\}),
\]
then $q$ is identified with both
$
f_1^{-1}(q)=[4,2,1,3]^T \in \mathcal{V}(\{1,2,3,4\})
$ and $
f_2^{-1}(q)=(143)\in S_4,
$
so that 4 has rank 1, 2 has rank 2, 1 has rank 3, and 3 has rank 4 in $q$. Moreover, if 
$\psi=(342)\in S_4$, then we can write
\[
\psi q=(342)[4,2,1,3]^T=[2,3,1,4]^T \quad \mbox{and}\quad q\rho_0=[4,2,1,3]^T(14)(23)=[3,1,2,4]^T,
\]
as well as
\[
\psi q=(342)(143)=(123) \quad \mbox{and}\quad q\rho_0=(143)(14)(23)=(132).
\]

Thus,  identifying preference relations with vectors makes computations easy and intuitive. 
On the other hand, identifying preference relations with permutations allows to transfer the group properties of $S_n$ 
to the products between preference relations and permutations. In particular, by associativity and cancellation laws, for every $\psi_1,\psi_2\in S_n$ and $\rho_1,\rho_2\in\{id, \rho_0\}$, we have that $\psi_1 q=\psi_2 q$ if and only if $\psi_1=\psi_2$;  $q\rho_1=q\rho_2$ if and only if $\rho_1=\rho_2$; $(\psi_2\psi_1)q=\psi_2(\psi_1 q)$; $q(\rho_1\rho_2)=(q\rho_1)\rho_2$; $(\psi_1 q)\rho_1=\psi_1(q \rho_1)$.

Given now $\psi\in S_n$ and $\rho\in\{id,\rho_0\}$, we finally emphasize that the above discussion makes the products  $\psi q$ and $q\rho$ have interesting interpretations. Indeed, if $q$ represents the preferences of a certain individual, then $\psi q$  represents the preferences that the individual would have if, for every $x\in N$, alternative $x$ were called $\psi(x)$; $q\rho$ represents the preferences that the individual would have if, for every $r\in \{1,\ldots,n\}$, the alternative whose rank is $r$ is moved to rank $\rho(r)$. As a consequence, even though both $\psi$ and $\rho$ belong to $S_n$
they have different meanings. Indeed, $\psi$ maps alternatives to alternatives, while $\rho$ maps ranks to ranks. Moreover, looking at $q$ as a permutation, we have that $q$ maps ranks to alternatives. In particular, the set $\{1,\ldots,n\}$ sometimes refers to the set of alternatives, sometimes to the set of ranks.
Although the context always allows to understand the right interpretation, along the paper we denote that set by $N$ in the first case, and by $\{1,\ldots,n\}$ in the second one.

\subsection{Preference profiles}\label{model}

\noindent From now on, let $h\in \mathbb{N}$ with $h\ge 2$ be fixed, and let $H=\{1,\ldots,h\}$ be the set of individuals.
A {\it preference  profile} is an element of $\mathcal{L}(N)^h$. The set $\mathcal{L}(N)^h$ is denoted by $\mathcal{P}$.
If $p\in\mathcal{P}$ and $i\in H$, the $i$-th component of $p$ is denoted by $p_i$ and represents the preferences of individual $i$. Any $p\in\mathcal{P}$ can be identified with the matrix whose $i$-th column is the column vector representing the $i$-th component of $p$.

Let us consider the groups $\Omega=\{id,\rho_0\}\le S_n$ and $G=S_h\times S_n\times \Omega$.
For every $(\varphi,\psi,\rho)\in G$ and $p\in \mathcal{P}$, define $p^{(\varphi,\psi,\rho)} \in \mathcal{P}$ as the preference profile such that, for every $i\in H$,
\begin{equation}\label{action}
\left(p^{(\varphi,\psi,\rho)}\right)_i=\psi p_{\varphi^{-1}(i)}\rho.
\end{equation}
Since we have given no meaning to $(p_i)^{(\varphi,\psi,\rho)}$ for a single preference relation $p_i\in \mathcal{L}(N)$, we will write the $i$-th component  $p^{(\varphi,\psi,\rho)}$ simply as $p^{(\varphi,\psi,\rho)}_i,$ instead of $\left(p^{(\varphi,\psi,\rho)}\right)_i$.

The preference profile $p^{(\varphi,\psi,\rho)}$ is then obtained by $p$ according to the following rules: for every $i\in H$, individual $i$ is renamed $\varphi(i)$; for every $x\in N$, alternative $x$ is renamed $\psi(x)$; for every $r\in\{1,\ldots,n\}$, alternatives whose rank is $r$ are moved to rank $\rho(r)$. For instance, if $n=3$, $h=5$ and
\[
p=
\left[
\begin{array}{ccccccc}
3 & 1 & 2 & 3 &2\\
2 & 2 & 1 & 2 &3\\
1 & 3 & 3 & 1 &1
\end{array}
\right],\quad \varphi=(134)(25),\quad \psi=(12),\quad \rho=\rho_0=(13),
\]
then we have
\[
p^{(\varphi,id,id)}=
\left[
\begin{array}{ccccccc}
3 & 2 &3  &2  &1\\
2 & 3 & 2 & 1 &2\\
1 & 1 & 1 & 3 & 3\\
\end{array}
\right],\quad p^{(id,\psi,id)}=
\left[
\begin{array}{ccccccc}
 3& 2 & 1 & 3 &1\\
 1& 1 &2  & 1 &3\\
 2& 3 & 3 & 2 &2\\
 \end{array}
\right],
\]
\[
p^{(id,id,\rho_0)}=
\left[
\begin{array}{ccccccc}
1 & 3 & 3 & 1 &1\\
2 & 2 & 1 & 2 &3\\
3 & 1 & 2 & 3 & 2 \\
 \end{array}
\right],\quad
p^{(\varphi,\psi,\rho_0)}=
\left[
\begin{array}{ccccccc}
2 & 2 & 2 & 3 &3\\
1 & 3 & 1 & 2 &1\\
3 & 1 & 3 & 1 & 2 \\
 \end{array}
\right].
\]
As it is easy to verify, if $n=2$, then $p^{(id,\rho_0,id)}=p^{(id,id,\rho_0)}$ for all $p\in\mathcal{P}$; if $n\ge 3$, then
there do not exist $\varphi\in S_h$ and $\psi\in S_n$ such that, for every $p\in\mathcal{P}$, $p^{(\varphi,\psi,id)}=p^{(id,id,\rho_0)}$. In other words,
top-down reversing preference profiles cannot be reduced, in general, to a change in individuals and alternatives names.

\subsection{Symmetric minimal majority rules}\label{sym-rul}

A {\it rule} (or social welfare function) is a function from $\mathcal{P}$ to $\mathcal{L}(N)$. Given a subgroup $U$ of $G$, 
we say that a rule $F$ is $U$-{\it symmetric} if, for every $p\in \mathcal{P}$ and $(\varphi,\psi,\rho)\in U$,
\[
F(p^{(\varphi,\psi,\rho)})=\psi F(p) \rho.
\]
The set of $U$-symmetric rules is denoted by $\mathcal{F}^U$. Note that if $U'\le U$, then $\mathcal{F}^{U}\subseteq  \mathcal{F}^{U'}$.

 The concept of symmetry with respect to a subgroup $U$ of $G$ includes some classical requirements for rules. For instance,  a rule
$F$ is anonymous if and only if $F\in\mathcal{F}^{S_h\times \{id\}\times \{id\}}$; it is neutral if and only if $F\in\mathcal{F}^{\{id\}\times S_n\times \{id\}}$; it  is reversal symmetric if and only if $F\in\mathcal{F}^{\{id\}\times \{id\}\times \Omega}$. Moreover, as explained in Sections \ref{def-not} and \ref{sub-sub}, any combination of the principles of anonymity, neutrality (even in their weak versions involving subcommittees and subclasses) and reversal symmetry can also be described in terms of $U$-symmetry for a suitable choice of the subgroup $U$. 
Thus, we decided to work first in the general and unifying setting of $U$-symmetric rules, and
rephrase later the results obtained for $U$-symmetry in
more specific but better interpretable contexts.

Given $\nu \in \mathbb{N}\cap(h/2,h]$,  define, for every $p\in\mathcal{P}$, the set
\[
C_{\nu}(p)=\left\{q\in \mathcal{L}(N): \forall x,y\in N, |\{i\in H: x>_{p_i} y\}|\ge \nu\;\Rightarrow\; x>_{q} y\right\},
\]
that is, the set of preference relations having $x\in N$ preferred to $y\in N$ whenever, according to the preference profile $p$, at least $\nu$ individuals prefer $x$ to $y$. In other words, we have that $C_{\nu}(p)$ is the set of linear orders that are consistent with the principle of qualified majority with majority threshold equal to $\nu$ (briefly $\nu$-majority) applied to the preference profile $p$. For example, consider $h=9$, $n=3$ (so that $H=\{1,\ldots,9\}$ and $N=\{1,2,3\}$) and the preference profile
\begin{equation}\label{example-ref}
p=\left[
\begin{array}{ccccccccc}
1 & 1 & 2 & 2 & 2 & 3& 3& 3 & 3\\
2 & 2 & 1 & 3 & 3  & 1& 1& 1 &2\\
3 & 3 & 3 & 1 & 1  & 2& 2&  2&1 \\
 \end{array}
\right].
\end{equation}
A simple check shows that
\[
\begin{array}{lll}
|\{i\in H: 1>_{p_i} 2\}|=5, &\quad |\{i\in H: 1>_{p_i} 3\}|=3,&\quad |\{i\in H: 2>_{p_i} 3\}|=5,\\
\vspace{-2mm}\\
|\{i\in H: 2>_{p_i} 1\}|=4, &\quad |\{i\in H: 3>_{p_i} 1\}|=6,&\quad |\{i\in H: 3>_{p_i} 2\}|=4.\\
\end{array}
\]
It is now immediate to compute, for every majority threshold $\nu\in\mathbb{N}\cap(\frac{h}{2},h]=\{5,6,7,8,9\}$, the set $C_{\nu}(p)$. Indeed, $C_5(p)$ is the set of linear orders such that 1 is ranked above 2, 2  above 3, and 3 above 1, that is, $C_5(p)=\varnothing$; $C_6(p)$ is the set of linear orders such that 3 is ranked above 1, that is, $C_6(p)=\{[3,1,2]^T,[3,2,1]^T,[2,3,1]^T\}$; $C_7(p)=C_8(p)=C_9(p)=\mathcal{L}(N)$ because, for every $\nu\in \{7,8,9\}$ and $x,y\in N$,  we have $|\{i\in H: x>_{p_i} y\}|<\nu$, that is,
the $\nu$-majority principle applied to $p$ does not generate any constraint.

Of course, if $\nu,\nu' \in \mathbb{N}\cap (h/2,h]$ with $\nu\le \nu',$ then we have $C_{\nu}(p)\subseteq C_{\nu'}(p)$ for all $p\in\mathcal{P}$.
It is also known that\footnote{See, for instance, Propositions 6 and 7 in Bubboloni and Gori (2014).} $C_\nu(p)\neq \varnothing$ for all $p\in \mathcal{P}$ if and only if
$\nu> \frac{n-1}{n}h$.
For every $p\in\mathcal{P}$, define now
\[
\nu(p)=\min\{\nu \in \mathbb{N}\cap (h/2,h]: C_{\nu}(p)	\neq \varnothing\},
\]
and note that $\nu(p)$ is well defined as, for every $p\in\mathcal{P}$, $C_{h}(p)\neq\varnothing$.
 A rule $F$ is said to be a {\it minimal majority rule} if, for every $p\in \mathcal{P}$, $F(p)\in C_{\nu(p)}(p)$.  For instance, the preference profile $p$ defined in \eqref{example-ref} is such that $\nu(p)=6$, so that if $F$ is a minimal majority rule, then we have $F(p)\in C_{\nu(p)}(p)=C_6(p)$.
We denote the set of minimal majority rules by $\mathcal{F}_{\min}$. Of course, the set  $\mathcal{F}_{\min}$  is nonempty  by definition.
It is worth noting an interesting connection between minimal majority rules and the well-known method of simple majority decision, that is, the function $S:\mathcal{P}\to \mathcal{R}(N)$ defined, for every $p\in\mathcal{P}$, as
\begin{equation}\label{smr}
S(p)=\left\{(x,y)\in N^2: |\{i\in H: x\ge_{p_i}y\}|\ge h/2\right\}.
\end{equation}
Indeed, recall that the image of $S$ is contained in the set of complete relations on $N$ but, except for $h$ odd and $n=2$, it is not contained in $\mathcal{L}(N)$.
Then, asking for linear orders as social outcomes, $S$ is significant only on the set $\mathcal{P}_S=\{p\in\mathcal{P}:S(p)\in \mathcal{L}(N)\}$.
The following proposition shows that each minimal majority rule can be seen as an extension to the whole set $\mathcal{P}$ of the function $S_{\mid \mathcal{P}_S}$.

\begin{proposition}\label{simple}
Let $p\in \mathcal{P}_S$. Then, $\nu(p)=\left\lfloor \frac{h+1}{2}\right\rfloor$ and $C_{\nu(p)}(p)=\{S(p)\}$. In particular, for every $F\in\mathcal{F}_{\min}$, we have that $F(p)=S(p)$.
\end{proposition}

\begin{proof}
Consider $p\in\mathcal{P}_S$, so that $S(p)\in\mathcal{L}(N)$, and define $\nu^*=\min\{\nu\in\mathbb{N}: \nu>h/2\}=\left\lfloor \frac{h+1}{2}\right\rfloor$. We claim that $C_{\nu^*}(p)=\{S(p)\}$.
Observe first that
$S(p)\in C_{\nu^*}(p)$. Indeed, let $x,y\in N$ such that $|\{i\in H: x>_{p_i}y\}|\ge\nu^*$, then $x\neq y$ and
$|\{i\in H: x>_{p_i}y\}|>h/2$, so that  $x>_{S(p)}y$. In order to show that $q\in C_{\nu^*}(p)$ implies $q=S(p)$, it is enough to prove that, for every $x,y\in N$, $x>_{S(p)}y$ implies $x>_q y$. Let  $x,y\in N$ such that $x>_{S(p)}y$. Thus, $x\neq y$ and
$|\{i\in H: x\ge_{p_i}y\}|\ge h/2$. Since $x\neq y$ and $p_i\in\mathcal{L}(N)$ for all $i\in H$, we get $|\{i\in H: x>_{p_i}y\}|\ge h/2$. However, we cannot have $|\{i\in H: x>_{p_i}y\}|= h/2$, since otherwise we should also have $(y,x)\in S(p)$, against $x>_{S(p)} y$. Thus, $|\{i\in H: x>_{p_i}y\}|> h/2$, so that $|\{i\in H: x>_{p_i}y\}|\ge \nu^*$. As a consequence, we have that $x>_q y$.

Since obviously $\nu(p)\ge \nu^*$, from the equality  $C_{\nu^*}(p)=\{S(p)\}$, we get that $\nu(p)=\nu^*$ and  $C_{\nu(p)}(p)=\{S(p)\}$. From $F\in\mathcal{F}_{\min}$, we finally deduce that $F(p)\in C_{\nu(p)}(p)=\{S(p)\}$, that is, $F(p)=S(p)$.
\end{proof}

Let us finally define, for every $U\le G$, the set  $\mathcal{F}^{U}_{\min}=\mathcal{F}^{U}\cap \mathcal{F}_{\min}$ of the {\it $U$-symmetric minimal majority rules}.
We are going to study under which conditions on the subgroup $U$
the sets $\mathcal{F}^U$ and $\mathcal{F}^U_{\min}$ are nonempty. Indeed, after having introduced in Section \ref{def-not} some fundamental algebraic tools, in Section \ref{ex-res}
we find out a condition on $U$, called regularity, that is equivalent to both $\mathcal{F}^U\neq\varnothing$ and $\mathcal{F}^U_{\min}\neq\varnothing$ (Theorems \ref{main} and \ref{mainmain}).

\section{Actions on the set of preference profiles}\label{def-not}

The next proposition, which generalizes Proposition 1 in Bubboloni and Gori (2014), shows that any subgroup $U$ of $G$ naturally acts on the set of preference profiles $\mathcal{P}$. That result is rich of consequences as it allows to exploit many general facts from group theory.

\begin{proposition}\label{action-l}
Let $U\le G$. Then:
\begin{itemize}
	\item[i)] for every $p\in\mathcal{P}$ and $(\varphi_1,\psi_1,\rho_1),(\varphi_2,\psi_2,\rho_2)\in U$, we have
\begin{equation}\label{action-e}
p^{\,(\varphi_1\varphi_2,\psi_1\psi_2,\rho_1\rho_2)}= \left(p^{\,(\varphi_2,\psi_2,\rho_2)}\right)^{(\varphi_1,\psi_1,\rho_1)};
\end{equation}
	\item[ii)] the function $f:U\to \mathrm{Sym}(\mathcal{P})$ defined, for every $(\varphi,\psi,\rho)\in U$, as
\[
f(\varphi,\psi,\rho):\mathcal{P}\to\mathcal{P},\quad p\mapsto p^{(\varphi,\psi,\rho)},
\]
is well posed and it is an action of the group $U$ on the set $\mathcal{P}$.
\end{itemize}
\end{proposition}

\begin{proof}
i) Fix $p\in  \mathcal{P}$ and $(\varphi_1,\psi_1,\rho_1),\, (\varphi_2,\psi_2,\rho_2)\in U$.
For every  $i\in H$, by  \eqref{action},  we have
\[
 p^{\,(\varphi_1\varphi_2,\psi_1\psi_2,\rho_1\rho_2)}_i =\psi_1\psi_2 p_{(\varphi_1\varphi_2)^{-1}(i)}\rho_1\rho_2,
 \]
and also, recalling that $\Omega$ is abelian,
\[
\left(p^{\,(\varphi_2,\psi_2,\rho_2)}\right)^{(\varphi_1,\psi_1,\rho_1)}_i=\psi_1 \left(p^{\,(\varphi_2,\psi_2,\rho_2)}\right)_{\varphi_1^{-1}(i)}\rho_1
=\psi_1\psi_2 p_{\varphi_2^{-1}(\varphi_1^{-1}(i))}\rho_2 \rho_1=\psi_1\psi_2 p_{(\varphi_1\varphi_2)^{-1}(i)}\rho_1\rho_2.
\]

ii) Fix $(\varphi,\psi,\rho)\in U$ and prove that $f(\varphi,\psi,\rho)\in \mathrm{Sym}(\mathcal{P})$. Since $\mathcal{P}$ is finite, it is enough to show that $f(\varphi,\psi,\rho)$ is surjective. Consider then $p\in\mathcal{P}$ and simply observe that, by \eqref{action} and \eqref{action-e},
\[
f(\varphi,\psi,\rho)\left(p^{\,(\varphi^{-1},\psi^{-1},\rho^{-1})}\right)=\left(p^{\,(\varphi^{-1},\psi^{-1},\rho^{-1})}\right)^{(\varphi,\psi,\rho)}=p^{(id,id,id)}=p.
\]
Since by \eqref{action-e} we also have that, for every $(\varphi_1,\psi_1,\rho_1),(\varphi_2,\psi_2,\rho_2)\in U$,
\begin{equation}\label{action-e-2}
f((\varphi_1,\psi_1,\rho_1) (\varphi_2,\psi_2,\rho_2))=f(\varphi_1,\psi_1,\rho_1)f(\varphi_2,\psi_2,\rho_2),
\end{equation}
we get that $f$ is an action of $U$ on $\mathcal{P}$.
\end{proof}

Thanks to the fact that the function $f$ defined in Proposition \ref{action-l} is an action, we can use in our context
notation and results concerning the action of a group on a set.
For every $p\in\mathcal{P}$, the set $\{p^g\in \mathcal{P}: g\in U\}$ is called the $U$-orbit of $p$ and is denoted by $p^U$. It is well known that
the set  $\mathcal{P}^U=\{p^U:p\in\mathcal{P}\}$ of the $U$-orbits is a partition\footnote{A partition of a nonempty set $X$ is a set of nonempty pairwise disjoint subsets of $X$ whose union is $X$.} of $\mathcal{P}$. We denote the order of $\mathcal{P}^U$ by $R(U)$.
Any vector $(p^j)_{j=1}^{R(U)}\in\mathcal{P}^{R(U)}$ such that $\mathcal{P}^U=\{p^{j\,U} : j\in\{1,\ldots,R(U)\}\}$,
is called a system of representatives of the $U$-orbits. The set of the systems of representatives of the $U$-orbits is nonempty and denoted by $\mathfrak{S}(U)$.
For every $p\in\mathcal{P}$, the stabilizer of $p$ in $U$ is the subgroup of $U$ defined by
\[
\mathrm{Stab}_U(p)=\{g\in U : p^g=p \},
\]
and it is well known that
\begin{equation}\label{orbord}
|p^U|=\displaystyle{\frac{|U|}{|\mathrm{Stab}_U(p)|}}.
\end{equation}
From Proposition \ref{action-l}, we gain the following simple but expressive result.

\begin{proposition}\label{U,V} Let $U,V\leq G$. Then $\mathcal{F}^U\cap \mathcal{F}^V=\mathcal{F}^{\langle U,V\rangle}.$
\end{proposition}

\begin{proof}  Since $\langle U,V\rangle$ contains both $U$ and $V$, we get $\mathcal{F}^{\langle U,V\rangle}\subseteq \mathcal{F}^U\cap \mathcal{F}^V.$ Let us now fix $F\in\mathcal{F}^U\cap \mathcal{F}^V$ and prove  $F\in \mathcal{F}^{\langle U,V\rangle}.$ Note first  that by the definition of generated subgroup, for every $g\in \langle U,V\rangle,$  there exists $k\in \mathbb{N}$ and $g_1,\ldots,g_k\in U\cup V$ such that $g=g_1\cdots g_k$. Let us define then, for every  $k\in \mathbb{N}$, the set $\langle U,V\rangle_k$ of the elements in $\langle U,V\rangle$ that can be written as product of $k$ elements of $U\cup V$. Then, we obtain $F\in \mathcal{F}^{\langle U,V\rangle}$ showing that, for every $k\in \mathbb{N}$, $p\in\mathcal{P}$ and $g=(\varphi, \psi, \rho)\in \langle U,V\rangle_k$, we have
$F(p^{(\varphi, \psi, \rho)})=\psi F(p)\rho$.
That can be easily proved by induction on $k$ using \eqref{action-e} and recalling that $\Omega$ is abelian.
\end{proof}

Proposition \ref{U,V} has interesting consequences. For instance, it implies that if $F$ is a rule, then $F$ is anonymous and neutral if and only if $F\in\mathcal{F}^{S_h\times S_n\times \{id\}}$; $F$ is anonymous and reversal symmetric if and only if $F\in\mathcal{F}^{S_h\times \{id\}\times \Omega}$; $F$ is neutral and reversal symmetric if and only if $F\in\mathcal{F}^{\{id\}\times S_n\times \Omega}$; $F$ is anonymous, neutral and reversal symmetric if and only if $F\in\mathcal{F}^{G}$.

\section{Existence results}\label{ex-res}
 In this section we study under which conditions on $U$ the sets $\mathcal{F}^U$ and $\mathcal{F}^U_{\min}$ are nonempty. Throughout the section, $U$ is a fixed subgroup of $G$.

\subsection{$U$-symmetric rules}

For every $p\in\mathcal{P}$, define the set
\[
S^U_1(p)=\left\{q\in \mathcal{L}(N): \forall (\varphi,\psi,\rho)\in \mathrm{Stab}_U(p), \psi q \rho=q\right\}.
\]

\begin{lemma}\label{anr-p-2}
If $F\in\mathcal{F}^{U}$, then, for every $p\in\mathcal{P}$, $F(p)\in S^U_1(p)$.
\end{lemma}

\begin{proof}Let $p\in\mathcal{P}$ and $(\varphi,\psi,\rho)\in \mathrm{Stab}_U(p)$. Then $p=p^{(\varphi,\psi,\rho)}$ and so
$F(p)=F(p^{(\varphi,\psi,\rho)})=\psi F(p)\rho,$ which says $F(p)\in S^U_1(p)$.
\end{proof}

\begin{proposition}\label{existence-A-N-R}
For every $(p^j)_{j=1}^{R(U)}\in\mathfrak{S}(U)$ and $(q_j)_{j=1}^{R(U)}\in \times_{j=1}^{R(U)} S^U_1(p^j)$, there exists a unique $F\in\mathcal{F}^{U}$ such that, for every $j\in\{1,\ldots, R(U)\}$, $F(p^j)=q_j$.
\end{proposition}

\begin{proof}Let $(p^j)_{j=1}^{R(U)}\in\mathfrak{S}(U)$, $(q_j)_{j=1}^{R(U)}\in  \times_{j=1}^{R(U)} S^U_1(p^j)$ and set $J=\{1,\ldots,R(U)\}$. Since $\{p^{j\,U} : j\in J\}$ is a partition of $\mathcal{P}$,
given $p\in \mathcal{P},$ there exist a unique $j\in J $ and $(\varphi,\psi,\rho)\in U$ such that $p=p^{j\,(\varphi,\psi,\rho)}$. Note that, if for some $j\in J $ there exist $(\varphi_1,\psi_1,\rho_1),(\varphi_2,\psi_2,\rho_2)\in U$ such that $p^{j\,(\varphi_1,\psi_1,\rho_1)}=p^{j\,(\varphi_2,\psi_2,\rho_2)}$, then $\psi_1 q_{j} \rho_1 =\psi_2 q_{j} \rho_2$. Indeed, by (\ref{action-e}), we have that $p^{j\,(\varphi_1,\psi_1,\rho_1)}=p^{j\,(\varphi_2,\psi_2,\rho_2)}$ implies
$(\varphi^{-1}_2\varphi_1,\psi_2^{-1}\psi_1,\rho_2^{-1}\rho_1)\in \mathrm{Stab}_U(p^{j}).$
Since $q_{j}\in S^U_1(p^{j})$ and $\Omega$ is abelian, we have that
$q_{j}=\psi^{-1}_2\psi_1 q_{j} \rho_2^{-1}\rho_1=\psi^{-1}_2\psi_1 q_{j} \rho_1\rho_2^{-1}$, and thus $\psi_1 q_{j} \rho_1 =\psi_2 q_{j} \rho_2$.

As a consequence, the rule $F$ defined, for every $p\in\mathcal{P}$, as $F(p)=\psi q_j \rho,$ where $j\in J $ and $(\varphi,\psi,\rho)\in U$ are such that $p=p^{j\,(\varphi,\psi,\rho)}$, is well defined. Moreover, for every $j\in J$, $F(p^j)=q_j$.
Let us prove that $F\in \mathcal{F}^{U}$. Consider $p\in \mathcal{P}$ and $(\varphi,\psi,\rho)\in U$ and let $p=p^{j\,(\varphi_1,\psi_1,\rho_1)}$ for some $j\in J$ and
$(\varphi_1,\psi_1,\rho_1)\in U$.
By the definition of $F$ and by \eqref{action-e},  we conclude that
\[
F(p^{(\varphi,\psi,\rho)})=F\left(\left(p^{j\,(\varphi_1,\psi_1,\rho_1)}\right)^{(\varphi,\psi,\rho)}\right)
=F(p^{j\,(\varphi\varphi_1,\psi\psi_1,\rho\rho_1)})=(\psi\psi_1)q_j(\rho\rho_1)
\]
\[
=(\psi\psi_1)q_j(\rho_1\rho)=\psi(\psi_1q_j\rho_1)\rho=\psi F(p^{j\,(\varphi_1,\psi_1,\rho_1)})\rho=\psi F(p)\rho.
\]
In order to prove the uniqueness of $F$, it suffices to note that if $F'\in\mathcal{F}^{U}$ is such that, for every $j\in J$, $F'(p^j)=q_j$, then  $F'(p^{j\,(\varphi,\psi,\rho)})=\psi q_j\rho=F(p^{j\,(\varphi,\psi,\rho)})$ for all $j\in J$ and $(\varphi,\psi,\rho)\in U$. Thus, for every $p\in\mathcal{P}$, $F'(p)=F(p)$.
\end{proof}

Given $(p^j)_{j=1}^{R(U)}\in\mathfrak{S}(U)$ and $(q_j)_{j=1}^{R(U)}\in \times_{j=1}^{R(U)} S^U_1(p^j)$, denote by $F\left[(p^j)_{j=1}^{R(U)},(q_j)_{j=1}^{R(U)}\right]$ the unique $F\in\mathcal{F}^{U}$ such that, for every $j\in\{1,\ldots, R(U)\}$, $F(p^j)=q_j$.
The next result, which is an immediate consequence of Lemma \ref{anr-p-2} and Proposition \ref{existence-A-N-R}, provides a formula to count the elements in $\mathcal{F}^U$, when a system of representatives is known. That formula is important under a theoretical perspective (see, for instance, the proof of Theorem \ref{main}) but it can also be useful in practical situations as shown in Section \ref{example}.

\begin{proposition}\label{anr}
Let $(p^j)_{j=1}^{R(U)}\in\mathfrak{S}(U)$. Then the function
\begin{equation}\label{function}
f:\times_{j=1}^{R(U)}  S^U_1(p^j)	\to \mathcal{F}^{U}, \quad  f\left((q_j)_{j=1}^{R(U)}\right)=F\left[(p^j)_{j=1}^{R(U)},(q_j)_{j=1}^{R(U)}\right]
\end{equation}
is bijective. In particular, $|\mathcal{F}^{U}|= \prod_{j=1}^{R(U)} | S^U_1(p^j)|$.
\end{proposition}

 Let us introduce now a crucial definition.
A subgroup $U$  of $G$ is said to be {\it regular} if, for every $p\in\mathcal{P}$,
\begin{equation}\label{property}
\begin{array}{c}
\mbox{there exists }\psi_*\in S_n\mbox{ conjugate to }\rho_0\mbox{ such that}\\
\vspace{-2mm}\\
\mathrm{Stab}_U(p)\subseteq \left(S_h\times \{id\}\times \{id\}\right)\cup \left( S_h\times \{\psi_*\}\times \{\rho_0\}\right).\\
\end{array}
\end{equation}
 Note that, within our notation, two permutations $\sigma_1,\sigma_2\in S_n$ are conjugate if there exists $u\in S_n$ such that $\sigma_1=u\sigma_2 u^{-1}.$
If $U$ is regular and $W\leq U$, then $W$ is regular too, because $\mathrm{Stab}_W(p)=W\cap \mathrm{Stab}_U(p)$. In particular, $G$ is regular if and only if each subgroup of $G$ is regular.
The next theorem shows the deep impact of the concept of regular subgroup in our research.

\begin{theorem}\label{main}
$\mathcal{F}^U\neq\varnothing$ if and only if $U$ is regular. Moreover, if $U$ is regular, then
$\left(2^{\lfloor \frac{n}{2}\rfloor} \lfloor \frac{n}{2}\rfloor !\right)^{R(U)}$ divides $|\mathcal{F}^U|$.
\end{theorem}

\begin{proof}
Assume that $\mathcal{F}^U\neq\varnothing$ and pick $F\in \mathcal{F}^U$. Fix $p\in\mathcal{P}$ and define $\psi_*=F(p)\rho_0F(p)^{-1}$. Given $(\varphi,\psi,\rho)\in \mathrm{Stab}_U(p)$, let us prove that $\rho=id$ implies $\psi=id$, and $\rho=\rho_0$ implies $\psi=\psi_*$.
Observe that $F(p)=F(p^{(\varphi,\psi,\rho)})=\psi F(p)\rho.$ As a consequence, if $\rho=id$ then we get $F(p)=\psi F(p)$ and thus $\psi=id$; if $\rho=\rho_0$ we get $F(p)=\psi F(p)\rho_0$ and thus, due to $|\rho_0|=2,$ we find that $\psi=F(p)\rho_0F(p)^{-1}=\psi_*$.

 Next assume that $U$ is regular and fix $p\in\mathcal{P}$. Let $\psi_*=u\rho_0u^{-1}$, for a suitable $u\in S_n,$ as in \eqref{property}. We show that
 \[
S_1^U(p)
=\left\{
\begin{array}{ll}
\mathcal{L}(N) & \mbox{if }\;\mathrm{Stab}_U(p)\leq S_h\times\{id\}\times\{id\}\\
\vspace{-2mm}\\
u  C_{S_n}(\rho_0)     & \mbox{if }\;\mathrm{Stab}_U(p)\not\leq S_h\times\{id\}\times\{id\}.\\
\end{array}
\right.
\]
The first fact is a trivial consequence of the definition of $S_1^U(p)$. Assume now that there exists $(\varphi_*,\psi_*,\rho_0)\in \mathrm{Stab}_U(p).$
By the regularity of $U$, the only elements of $\mathrm{Stab}_U(p)$ affecting $S_1^U(p)$ are those belonging to $S_h\times\{\psi_*\}\times\{\rho_0\}$,
so that $S_1^U(p)=\{q\in \mathcal{L}(N): \psi_* q \rho_0=q\}$. However,
\[
q\in S_1^U(p) \quad\Leftrightarrow \quad \psi_*= q\rho_0q^{-1} \quad\Leftrightarrow \quad \rho_0(u^{-1}q)=(u^{-1}q)\rho_0
\]
\[
\Leftrightarrow \quad u^{-1}q\in C_{S_n}(\rho_0)\quad\Leftrightarrow \quad q\in uC_{S_n}(\rho_0),
\]
which means $S_1^U(p)=uC_{S_n}(\rho_0)$.

Since $|\mathcal{L}(N)|=n!$ and it is well known that $|C_{S_n}(\rho_0)|=2^{\lfloor \frac{n}{2}\rfloor} \lfloor \frac{n}{2}\rfloor !,$ we also get
 \[
|S_1^U(p)|
=\left\{
\begin{array}{ll}
n! & \mbox{if }\;\mathrm{Stab}_U(p)\leq S_h\times\{id\}\times\{id\}\\
\vspace{-2mm}\\
  2^{\lfloor \frac{n}{2}\rfloor} \lfloor \frac{n}{2}\rfloor !     &   \mbox{if }\;\mathrm{Stab}_U(p)\not\leq S_h\times\{id\}\times\{id\}.\\
\end{array}\right.
\]
Thus, by Proposition \ref{anr},
$|\mathcal{F}^U|\ge 1$.
 Moreover, since $C_{S_n}(\rho_0)\leq S_n$, we have that $|C_{S_n}(\rho_0)|$ divides $n!$ and then
 $\left(2^{\lfloor \frac{n}{2}\rfloor} \lfloor \frac{n}{2}\rfloor !\right)^{R(U)}$ divides $|\mathcal{F}^U|$.
\end{proof}

\subsection {$U$-symmetric minimal majority rules}

We start our study of  $\mathcal{F}_{\min}^U$ with a preliminary lemma dealing with the behaviour of the set $C_{\nu(p)}(p)$ with respect to the action of $U.$
\begin{lemma}\label{c-nu}
Let $\nu\in \mathbb{N}\cap (h/2,h]$, $p\in\mathcal{P}$ and $(\varphi,\psi,\rho)\in G$. Then $C_{\nu}(p^{(\varphi,\psi,\rho)})=\psi C_{\nu}(p)\rho$,
$\nu(p^{(\varphi,\psi,\rho)})=\nu(p)$ and $C_{\nu(p^{(\varphi,\psi,\rho)})}(p^{(\varphi,\psi,\rho)})=\psi C_{\nu(p)}(p)\rho$.
\end{lemma}

\begin{proof} By Lemma 10 in Bubboloni and Gori (2014) we know that  $C_{\nu}(p^{(\varphi,\psi,id)})=\psi C_{\nu}(p)$ for all $p\in\mathcal{P}$ and all $(\varphi,\psi,id)\in G$. Since, by \eqref{action-e}, $C_{\nu}(p^{(\varphi,\psi,\rho_0)})=C_{\nu}(\left(p^{(\varphi,\psi,id)}\right)^{(id,id,\rho_0)}),$ we can prove the first part of the statement showing that, for every $p\in\mathcal{P}$, we have
\[
C_{\nu}(p^{(id,id,\rho_0)})=C_{\nu}(p)\rho_0.
\]
However, due to $p^{(id,id,\rho_0)}_i=p_i\rho_0$ and recalling that $|\rho_0|=2$, we immediately have
\[
\begin{array}{c}
C_{\nu}(p^{(id,id,\rho_0)})=\{q\in \mathcal{L}(N): \forall x,y\in N, |\{i\in H: x>_{p_i\rho_0} y\}|\ge \nu\Rightarrow x>_{q} y\}\\
\vspace{-2mm}\\
= \{q\in \mathcal{L}(N): \forall x,y\in N, |\{i\in H: y>_{p_i} x\}|\ge \nu\Rightarrow y>_{q\rho_0} x\}\\
\vspace{-2mm}\\
= \{q_1\rho_0\in \mathcal{L}(N): \forall x,y\in N, |\{i\in H: y>_{p_i} x\}|\ge \nu\Rightarrow y>_{q_1} x\}=C_{\nu}(p)\rho_0.\\
\end{array}
\]
In order to complete the proof, note that $|C_{\nu}(p^{(\varphi,\psi,\rho)})|=| C_{\nu}(p)|$ and thus $C_{\nu}(p^{(\varphi,\psi,\rho)})\neq\varnothing$ if and only if $C_{\nu}(p)\neq\varnothing$, that is, $\nu(p^{(\varphi,\psi,\rho)})=\nu(p)$. It also follows that $C_{\nu(p^{(\varphi,\psi,\rho)})}(p^{(\varphi,\psi,\rho)})=C_{\nu(p)}(p^{(\varphi,\psi,\rho)})=\psi C_{\nu(p)}(p)\rho$.
\end{proof}

For every $p\in\mathcal{P}$, define  the set
\[
S^U_2(p)=S^U_1(p)\cap C_{\nu(p)}(p).
\]
\begin{lemma}\label{anrmax-p-2}
If $F\in\mathcal{F}^{U}_{\min}$, then, for every $p\in\mathcal{P}$, $F(p)\in S^U_2(p)$.
\end{lemma}

\begin{proof}
Let $p\in\mathcal{P}$. Since $F\in\mathcal{F}^{U}$, by Lemma \ref{anr-p-2}, we know that $F(p)\in S^U_1(p)$. Moreover, as $F\in \mathcal{F}_{\min}$, we also have that $F(p)\in C_{\nu(p)}(p)$. Thus $F(p)\in S^U_2(p)$.
\end{proof}

 Proposition \ref{n-an-nu} below is analogous to Proposition \ref{anr}. We stress that it is a fundamental tool to prove Theorem \ref{mainmain}.

\begin{proposition}\label{n-an-nu}
Let  $(p^j)_{j=1}^{R(U)}\in\mathfrak{S}(U)$ and $f$ defined as in \eqref{function}. Then
\[
f\left(\times_{j=1}^{R(U)}S^U_2(p^j)\right)=\mathcal{F}_{\min}^{U}.
\]
In particular, $|\mathcal{F}^{U}_{\min}|=\prod_{j=1}^{R(U)}|S^U_2(p^j)|$.
\end{proposition}

\begin{proof} Set $S=\times_{j=1}^{R(U)}S^U_2(p^j)$ and $J=\{1\ldots,R(U)\}$.
In order to prove that $f(S)\subseteq \mathcal{F}^{U}_{\min}$, let $(q_j)_{j=1}^{R(U)}\in S$ and prove that $F=F\left[(p^j)_{j=1}^{R(U)},(q_j)_{j=1}^{R(U)}\right]\in \mathcal{F}^{U}_{\min}$.
By Proposition \ref{anr}, we immediately have that $F\in \mathcal{F}^{U}$. Consider now $p\in\mathcal{P}$.
Then there exist $j\in J$ and $(\varphi,\psi,\rho)\in U$ such that $p=p^{j\,(\varphi,\psi,\rho)}$. As, for every $j\in J$, we know that $q_j=F(p^j)\in C_{\nu(p^j)}(p^j)$,
by Lemma \ref{c-nu}, we also have
\[
F(p)=\psi q_j \rho\in \psi  C_{\nu(p^j)}(p^j)\rho = C_{\nu(p^{j\,(\varphi,\psi,\rho)})}(p^{j\,(\varphi,\psi,\rho)})=C_{\nu(p)}(p),
\]
as desired.

In order to prove that  $\mathcal{F}^{U}_{\min} \subseteq f(S)$, let $F\in \mathcal{F}^{U}_{\min}$ and define, for every $j\in J$, $q_j=F(p^j)$.  Then by Lemma \ref{anrmax-p-2}, we immediately have $(q_j)_{j=1}^{R(U)}\in S$ and since  $F=f\left((q_j)_{j=1}^{R(U)}\right)$ we get $F\in f(S)$.
\end{proof}

The following theorem emphasizes the importance of regular groups. Indeed, it shows that those groups are consistent not only  with the symmetry of rules, as established in Theorem \ref{main}, but also with the minimal majority principle. Its quite technical proof can be found in Section \ref{appendix}.

\begin{theorem}\label{mainmain}$\mathcal{F}^U_{\min}\neq\varnothing$ if and only if  $U$ is regular.
\end{theorem}

\section{Regular groups}\label{char}

 Due to Theorems \ref{main} and \ref{mainmain}, it is important to find some simple criteria to check whether a group is regular or not. In this section we characterize those groups via two properties that, as shown in Section \ref{sub-sub}, are simple to verify in some remarkable situations. To present that characterization result, we need to say something more about permutations.
\subsection{Orbits and types of permutations}

Fix $\sigma\in \mathrm{Sym}(X)$. For every $x\in X$, the $\sigma$-orbit of $x$ is defined as $x^{\langle\sigma\rangle}=\{\sigma^m(x)\in X: m\in \mathbb{N}\}.$
It is well known that $|x^{\langle\sigma\rangle}|=s$ if and only if $s=\min\{m\in\mathbb{N}: \sigma^m(x)=x\}$. 
 The set $O(\sigma)=\{x^{\langle\sigma\rangle}: x\in X\}$ of the $\sigma$-orbits is a partition of $X$, and we denote its order by $r(\sigma)$. A system of representatives of  the $\sigma$-orbits is a vector $(x_j)_{j=1}^{r(\sigma)}\in X^{r(\sigma)}$ such that $O(\sigma)=\{x_1^{\langle\sigma\rangle},\dots, x_{r(\sigma)}^{\langle\sigma\rangle}\}$. Note that
\begin{equation}\label{benpostoi}
X=\left\{\sigma^m(x_j)\in X: m\in\mathbb{N},\, j\in\{1,\dots,r(\sigma)\} \right\},
\end{equation}
and
\begin{equation}\label{benpostoii}
\begin{array}{c}
\mbox{for every }m,\ell\in\mathbb{N},\mbox{ and }j_1,j_2\in\{1,\dots, r(\sigma) \},\\
\sigma^m(x_{j_1})=\sigma^{\ell}(x_{j_2})\mbox{  if and only if }j_1=j_2\mbox{ and }|x_{j_1}^{\langle\sigma\rangle}| \mbox{ divides }\ell-m.
\end{array}
\end{equation}
A system of representatives of the $\sigma$-orbits $(x_j)_{j=1}^{r(\sigma)}\in X^{r(\sigma)}$ is called ordered if $|x_{j_1}^{\langle\sigma\rangle}|\geq |x_{j_2}^{\langle\sigma\rangle}|$ for all $j_1,j_2\in\{1,\dots,r(\sigma)\}$ with $j_1\le j_2.$

Given $k\in \mathbb{N}$, the set of partitions of $k$ is
\[
\mbox{$\Pi(k)=\bigcup_{r=1}^k\big\{\lambda=(\lambda_{j})_{j=1}^r\in\mathbb{N}^r: \sum_{j=1}^r\lambda_j=k, \lambda_1\ge \ldots\ge \lambda_r\big\}$}.
\]
Let us consider now the well known surjective function
\[
\mbox{$T:\bigcup_{k\in\mathbb{N}}S_k\to \bigcup_{k\in\mathbb{N}}\Pi(k)$},\quad \sigma\mapsto T(\sigma)=(T_j(\sigma))_{j=1}^{r(\sigma)}=(|x_j^{\langle\sigma\rangle}|)_{j=1}^{r(\sigma)},
\]
where $(x_j)_{j=1}^{r(\sigma)}$ is any ordered system of representatives of the $\sigma$-orbits. Note that $T$ is well defined since $T(\sigma)$, called the  type of $\sigma$, does not depend on the particular ordered system of representatives of the $\sigma$-orbits chosen.
Note also that if $\sigma\in S_k$, then $T(\sigma)$ belongs to $\Pi(k)$. Moreover, the number of  components equal to $1$ in the vector $T(\sigma)$ is equal to the number of fixed points of $\sigma$, and $|\sigma|= \mathrm{lcm }(T(\sigma)).$ For instance, if $\sigma=(123)(456)(78)\in S_9$, then $r(\sigma)=4$, an ordered system of representatives of the $\sigma$-orbits is $(1,4,7,9)\in\{1,\ldots,9\}^4$, the type of $\sigma$ is $T(\sigma)=(3,3,2,1)\in \Pi(9)$, and $|\sigma|=\mathrm{lcm }(3,3,2,1)=6.$

The  theoretic importance of the concept of type relies on the fact that two permutations are conjugate if and only if they have the same type. Looking at the specific purposes of the paper, we are going to see how checking the regularity of a group $U\le G$ reduces to check, for every $(\varphi,\psi,\rho)\in U$, some arithmetical properties of the order of $\psi$ and the type of $\varphi.$

Given  a prime $\pi$ and $\sigma\in\bigcup_{k\in\mathbb{N}}S_k$, we set $|\sigma|_\pi=\max \{\pi^a: a\in\mathbb{N}\cup\{0\},\   \pi^a\mid |\sigma|\}$.

\subsection{Characterization of regular groups}

\begin{theorem}\label{Regular}
Let $U\le G$. $U$ is regular if and only if the two following conditions are satisfied:
\begin{itemize}
\item[a)]  if $(\varphi,\psi,id)\in U$ is such that $\psi\neq id$ and $\pi$ is a prime with $|\psi|_\pi=\pi^a$ for some $a\in \mathbb{N},$ then $\pi^a\nmid \gcd(T(\varphi))$;
\item[b)]  if $(\varphi,\psi,\rho_0)\in U$  is such that $\psi^2=id$ and $\psi$ is not a conjugate of $\rho_0$, then $2\nmid \gcd(T(\varphi)).$
\end{itemize}

\end{theorem}
\begin{proof}
We first prove that conditions a) and b) are necessary for the regularity of $U.$

We begin showing that if a) does not hold, then $U$ is not regular.
By assumption there exist $g=(\varphi,\psi,id)\in U$ with $\psi\neq id$, a prime  $\pi$ and $a\in \mathbb{N}$ such that $\pi^a=|\psi|_\pi\mid  \gcd(T(\varphi)),$  which means that each $\varphi$-orbit has order divisible by $\pi^a.$
Consider the positive integer $m=|\psi|/\pi^a$ and let  $g^m=(\varphi^m,\psi^m,id)\in U.$ Set  $g^m=\hat{g}$,  $\varphi^m=\hat{\varphi}$ and $\psi^m=\hat{\psi}$. Since $\pi\nmid m$, each $\hat{\varphi}$-orbit has order divisible by $\pi^a,$ that is, $\pi^a\mid \gcd (T(\hat{\varphi}));$ moreover, by construction, $|\hat{\psi}|=\pi^a>1.$
 Let $(i_j)_{j=1}^{r(\hat{\varphi})}$ be an ordered system of representatives for the $\hat{\varphi}$-orbits, so that, by \eqref{benpostoi}, $H=\{\hat{\varphi}^k(i_j) : k\in\mathbb{N}, j\in\{1,\dots,r(\hat{\varphi})\} \}$. Let $p\in\mathcal{P}$ be defined by  $p_{\hat{\varphi}^k(i_j)}=\hat{\psi}^{k}$  for all $j\in\{1,\dots,r(\hat{\varphi})\}$ and $k\in\mathbb{N}$.  We show that the definition of $p$ is well posed. Let $\hat{\varphi}^k(i_{j_1})=\hat{\varphi}^{\ell}(i_{j_2}),$ for
  some $k,\ell\in\mathbb{N}$ and $j_1,j_2\in \{1,\dots,r(\hat{\varphi})\}$. Then, by \eqref{benpostoii}, we have $j_1=j_2$ and
 $T_{j_1}(\hat{\varphi}) \mid \ell-k$. So we also have $\pi^a\mid \ell-k$ and, since $|\hat{\psi}|=\pi^a$, we finally obtain
$\hat{\psi}^{k}=\hat{\psi}^{\ell}.$ We now see that $p^{\hat{g}}=p,$ that is,  for every
$i\in H, $ $p_{\hat{\varphi}(i)}=\hat{\psi} p_i.$ Pick $i\in H$; then there exist $j\in \{1,\dots,r(\hat{\varphi})\}$ and $k\in \mathbb{N}$ such that $i=\hat{\varphi}^k(i_j)$ and thus $p_{\hat{\varphi}(i)}=p_{\hat{\varphi}(\hat{\varphi}^k(i_j))}=p_{\hat{\varphi}^{k+1}(i_j)}= \hat{\psi}^{k+1}=\hat{\psi } p_{\hat{\varphi}^{k}(i_j)}=\hat{\psi} p_i$.
So we have $( \hat{\varphi},\hat{\psi},id)\in \mathrm{Stab}_U(p),$ with $\hat{\psi}\neq id,$ which implies that $U$ is not regular.

Next we show that if  b) does not hold, then  $U$ is not regular.
By assumption there exists $g=(\varphi,\psi,\rho_0)\in U$ with $\psi^2=id$ and $\psi$ not a conjugate of $\rho_0$ such that $2\mid \gcd(T(\varphi)).$  Let $(i_j)_{j=1}^{r(\varphi)}$ be an ordered system of representatives for the $\varphi$-orbits, so that, by \eqref{benpostoi},
$H=\{\varphi^k(i_j) : k\in\mathbb{N},j\in\{1,\dots,r(\varphi)\} \}$.
 Let $p\in\mathcal{P}$ be defined by $p_{\varphi^k(i_j)}=\psi^{k}\rho_0^{k},$ for all $j\in\{1,\dots,r(\varphi)\}$ and $k\in\mathbb{N}.$ Since $\psi^2=\rho_0^2=id,$ this  simply means $p_{\varphi^k(i_j)}=id$  for $k$ even and  $p_{\varphi^k(i_j)}=\psi\rho_0$  for $k$ odd. We show that the definition of $p$ is well posed. Let $\varphi^k(i_{j_1})=\varphi^{\ell}(i_{j_2}),$ for some $j_1, j_2\in \{1,\dots,r(\varphi)\}$ and some $k,\ell\in\mathbb{N}$. Then, by \eqref{benpostoii},
 $j_1=j_2$ and $T_{j_1}(\varphi) \mid \ell-k$, which implies $2\mid \ell-k$ and therefore
$\psi^{k}\rho_0^{k}=\psi^{\ell}\rho_0^{\ell}.$
We claim that $p^g=p$, that is,  for every
$i\in H, $ $p_{\varphi(i)}=\psi p_i\rho_0.$ Pick $i\in H$; then there exist $j\in \{1,\dots,r(\varphi)\}$ and $k\in \mathbb{N}$ such that $i=\varphi^k(i_j)$, and thus $p_{\varphi(i)}=p_{\varphi(\varphi^k(i_j))}=p_{\varphi^{k+1}(i_j)}= \psi^{k+1}\rho_0^{k+1}=\psi p_{\varphi^{k}(i_j)} \rho_0=\psi p_i\rho_0$.
So we have $(\varphi,\psi,\rho_0)\in \mathrm{Stab}_U(p),$ with $\psi$ not conjugate to $\rho_0$ and thus $U$ is not regular.

Let us prove now that conditions a) and b) are sufficient for the regularity of $U.$  First of all, we show that, for every $p\in\mathcal{P}$,
\begin{equation}\label{egly}
(\varphi,\psi,id)\in \mathrm{Stab}_U(p)\ \Rightarrow\  |\psi|\mid \gcd(T(\varphi)).
\end{equation}
 Namely, from $p^{(\varphi,\psi,id)}=p$ we get $p_{\varphi(i)}=\psi p_i$ for all $i\in H$ and thus also $p_{\varphi^k(i)}=\psi^k p_i$ for all $k\in \mathbb{N}$ and all $i\in H$. Let
 $(i_j)_{j=1}^{r(\varphi)}$ be an ordered system of representatives for the $\varphi$-orbits. Then $\varphi^{T_j(\varphi)}(i_j)=i_j$ and so $p_{i_j}=p_{\varphi^{T_j(\varphi)}(i_j)}=\psi^{T_j(\varphi)} p_{i_j},$ which says $\psi^{T_j(\varphi)}=id$. Therefore, for every $j\in \{1,\dots,r(\varphi)\}$, $|\psi|\mid T_j(\varphi)$, that is, $|\psi|\mid \gcd(T(\varphi)).$

Let now $p\in\mathcal{P}$ be fixed. In order to get the regularity of $U$, we first prove that, for every $(\varphi,\psi,id)\in \mathrm{Stab}_U(p)$, we have $\psi=id$. Consider then $g=(\varphi,\psi,id)\in \mathrm{Stab}_U(p)$ and assume by contradiction that $\psi\neq id$. Thus, there exists at least one prime $\pi$ with $|\psi|_\pi>1,$ say $|\psi|_\pi=\pi^a$ for some $a\in \mathbb{N}.$ Moreover, there exists $m\in \mathbb{N}$ such that $\pi\nmid m$ and $|\psi^m|=\pi^a$. Since $\mathrm{Stab}_U(p)$ is a subgroup of $U$, we have that $g^m=(\varphi^m,\psi^m,id)\in \mathrm{Stab}_U(p).$ Thus, by \eqref{egly}, $\pi^a=|\psi^m|\mid \gcd(T(\varphi^m)).$ Yet, it is easily observed that $ \gcd(T(\varphi^m))\mid  \gcd(T(\varphi))$ and so we also have  $\pi^a  \mid \gcd(T(\varphi))$, against condition a).
We finally need to show that there exists $\psi_*\in S_n$ conjugate to $\rho_0$ such that, for every $(\varphi,\psi,\rho_0)\in \mathrm{Stab}_U(p)$, $\psi=\psi_*$. Consider $g=(\varphi,\psi,\rho_0)\in \mathrm{Stab}_U(p)$ and first prove that $\psi$ is a conjugate of $\rho_0$. Note that $g^2=(\varphi^2,\psi^2,id)\in \mathrm{Stab}_U(p)$ and thus, by the previous case, $\psi^2=id$. Assume, by contradiction, that $\psi$ is not a conjugate of $\rho_0$. Since $g\in \mathrm{Stab}_U(p)\le U,$ we have that $p_{\varphi^k(i)}=\psi^k p_i\rho_0^k$ for all $k\in \mathbb{N}$ and all $i\in H$. Due to $\psi^2=\rho_0^2=id,$ that means $p_{\varphi^k(i)}=p_i$ for $k$ even and  $p_{\varphi^k(i)}=\psi p_i\rho_0$ for $k$ odd.
If there exists a $\varphi$-orbit $i^{\langle \varphi\rangle}$ of odd order $k$, then we have $p_i=p_{\varphi^k(i)}=\psi p_i\rho_0$ and therefore $\psi=p_i\rho_0p_i^{-1}$ is a conjugate of $\rho_0,$ against our assumption. So, for every $j\in \{1,\dots,r(\varphi)\},$ $T_j(\varphi)$ is even, which contradicts condition b).
Thus, we are left with proving that if $g=(\varphi,\psi,\rho_0),\  g'=(\varphi',\psi',\rho_0)\in \mathrm{Stab}_U(p)$ then $\psi=\psi'.$ This is immediately done noting that $gg'^{-1}=(\varphi\varphi'^{-1},\psi\psi'^{-1},id)\in \mathrm{Stab}_U(p)$ which, as already proved, implies $\psi\psi'^{-1}=id.$
\end{proof}

\section{Subcommittees and subclasses}\label{sub-sub}

In this section we focus on rules that are anonymous with respect to subcommittees, neutral with respect to subclasses and reversal symmetric.
To begin with, let us formalize those versions of the principles of anonymity and neutrality  in terms of $U$-symmetry.

Given $B=\{B_j\}_{j=1}^s$ a partition of $H$, we define
\[
V(B)=\left\{\varphi\in S_h : \varphi(B_j)=B_j\hbox{ for all } j\in\{1,\dots,s\} \right\},
\]
and given $C=\{C_k\}_{k=1}^t$ a partition of $N$, we define
\[
W(C)=\left\{\psi\in S_n :  \psi(C_k)=C_k \hbox{ for all }  k\in\{1,\dots,t\} \right\}.
\]
Note that $V(B)$ is a subgroup of $S_h$ and $W(C)$ is a subgroup of $S_n$.
Moreover,  $V(\{H\})=S_h$ and $W(\{N\})=S_n$.

 A rule is said to be anonymous with respect to a partition $B$ of $H$, briefly {\it $B$-anonymous}, if it is $V(B)\times \{id\}\times \{id\}$-symmetric.
A rule is said to be neutral with respect to a partition $C$ of $N$, briefly {\it $C$-neutral}, if it is $ \{id\}\times W(C)\times\{id\}$-symmetric.  Thus, referring to the discussion carried on in the introduction, if $B$ is interpreted as  the set of subcommittees in which $H$ is divided, then $B$-anonymous rules are those rules which do not distinguish among individuals belonging to the same subcommittee. Analogously, interpreting $C$ as the set of subclasses in which $N$ is divided, we have that $C$-neutral rules are those rules equally treating alternatives within each subclass.
Note also that, because of Proposition \ref{U,V}, a rule is $B$-anonymous and $C$-neutral if and only if it is  $V(B)\times W(C)\times \{id\}\,$-symmetric. Similarly,  a rule is $B$-anonymous, $C$-neutral and reversal symmetric if and only if it is  $V(B)\times W(C)\times \Omega\,$-symmetric.

Using Theorems \ref{mainmain} and \ref{Regular}, we can now prove the main result of the paper, that is, Theorem \ref{comm} below. It provides simple tests to check whether, given a partition of individuals into subcommittees and a partition of alternatives into subclasses, there exists a minimal majority rule which is anonymous with respect to the considered subcommittees, neutral with respect to the considered subclasses and possibly reversal symmetric.

\begin{lemma}\label{propertyU} Let $B=\{B_j\}_{j=1}^s$ be a partition of $H$ and $C=\{C_k\}_{k=1}^t$ be a partition of $N$. Then, for every $(\varphi,\psi,\rho)\in  V(B)\times W(C)\times \Omega$, we have
\[
\gcd(T(\varphi))\mid \gcd(|B_j|)_{j=1}^s\quad \mbox{and}\quad |\psi|\mid \mathrm{lcm}(|C_k|!)_{k=1}^t.
\]
\end{lemma}
\begin{proof}Let $(\varphi,\psi,\rho)\in V(B)\times W(C)\times \Omega$. Since $\varphi\in V(B)$, each element of $O(\varphi)$ is a subset of a suitable element of $B$, which immediately implies
$
\gcd(T(\varphi))\mid \gcd(|B_j|)_{j=1}^s.
$
On the other hand, since $\psi\in W(C)$,  we have that $\psi=u_1\cdots u_t$ for suitable  pairwise commuting permutations $u_1,\ldots,u_t\in S_n$ such that, for every $k\in\{1,\ldots,t\}$, $u_k$ fixes all the elements in $N\setminus C_k$, so that $|u_k|\mid |C_k|!$. As a consequence, we have that $
|\psi|=\mathrm{lcm}(|u_1|,\dots,|u_t|)$ and then $|\psi|\mid\mathrm{lcm}(|C_k|!)_{k=1}^t.
$
\end{proof}

\begin{theorem}\label{teonuovo}
 Let $B=\{B_j\}_{j=1}^s$ be a partition of $H$ and $C=\{C_k\}_{k=1}^t$ be a partition of $N$. Let $|C_{k^*}|=\max\{ |C_k| \}_{k=1}^t .$  Then:
 \begin{itemize}

 \item[i)] $V(B)\times W(C) \times \{id\}$ is regular if and only if
\begin{equation}\label{gcd2}
\gcd\left(\gcd(|B_j|)_{j=1}^s, \,|C_{k^*}|!\right)=1;
\end{equation}
\item[ii)] $V(B)\times W(C) \times \Omega$ is regular if and only if
\begin{equation}\label{gcd}
\gcd\left(\gcd(|B_j|)_{j=1}^s, \,\mathrm{lcm}(2,|C_{k^*}|!)\right)=1.
\end{equation}
\end{itemize}
\end{theorem}

\begin{proof}Let $U_1=V(B)\times W(C)\times \{id\}$ and $U_2=V(B)\times W(C)\times \Omega$.

We first prove statement ii).
Assume that condition \eqref{gcd} holds true and show that $U_2$ is regular.  By Theorem \ref{Regular}, we need to show that conditions a) and b) are satisfied. Let $(\varphi,\psi,id)\in U_2$ with $\psi\neq id$ and $\pi$ be a prime such that $|\psi|_\pi=\pi^a$ for some $a\in \mathbb{N}$. By contradiction, assume that
 $\pi^a\mid \gcd(T(\varphi)).$ By Lemma \ref{propertyU}, $\gcd(T(\varphi))\mid \gcd(|B_j|)_{j=1}^s$ and  $|\psi|\mid \mathrm{lcm}(|C_k|!)_{k=1}^t=|C_{k^*}|!$. In particular,  
$\pi\mid \gcd(|B_j|)_{j=1}^s$ and $\pi\mid |C_{k^*}|!$, so that $\pi\mid \gcd\left(\gcd(|B_j|)_{j=1}^s, \,\mathrm{lcm}(2,|C_{k^*}|!)\right)=1$ and the contradiction is found. Let now $(\varphi,\psi,\rho_0)\in U$  with $\psi^2=id$, $\psi$ not a conjugate of $\rho_0$ and, by contradiction, assume that $2\mid \gcd(T(\varphi)).$ Then by Lemma \ref{propertyU}, we get $2\mid \gcd(|B_j|)_{j=1}^s.$ Therefore, $2\mid \gcd\left(\gcd(|B_j|)_{j=1}^s, \,\mathrm{lcm}(2,|C_{k^*}|!)\right)=1$, a  contradiction.

Assume next that there exists a prime $\pi$ such that $\pi\mid \gcd\left(\gcd(|B_j|)_{j=1}^s, \,\mathrm{lcm}(2,|C_{k^*}|!)\right),$ and show that $U_2$ is not regular. Note that $\pi\mid |B_j|$ for all $j\in\{1,\ldots,s\}$.
If $|C_{k^*}|=1$, then $\pi=2$ and we show that condition b) in Theorem \ref{Regular}, fails. Indeed, choose $\varphi\in S_h$  cyclically permuting all the elements in $B_j$ for all $j\in\{1,\ldots,s\}$,
and consider $(\varphi, id,\rho_0)\in U_2$. We have that $id^2=id$ and $id$ is not conjugate of $\rho_0,$ because $|\rho_0|=2$. However, by the definition of $\varphi$, $2\mid \gcd(T(\varphi))=\gcd(|B_j|)_{j=1}^s.$ If $|C_{k^*}|\ge 2$, then $\pi \mid |C_{k^*}|!$, that is, $\pi \le |C_{k^*}|$.
We show that condition a) in Theorem \ref{Regular}, fails. Choose $\varphi\in S_h$ cyclically permuting all the elements in $B_j$ for all $j\in\{1,\ldots,s\}$, and  $\psi\in S_n$  acting as a cycle of length $\pi$ on the set $C_{k^*}$ and leaving fixed any other element in $N$. Clearly, $\psi(C_{k})=C_{k}$ for all $k\in\{1,\dots,t\}$, so that $(\varphi,\psi,id)\in U_2$ and $\pi=|\psi|=|\psi|_\pi\mid \gcd(T(\varphi))=\gcd(|B_j|)_{j=1}^s.$

We now prove statement i). Assume that condition \eqref{gcd2} holds true and show that $U_1$ is regular. If $|C_{k^*}|=1$, then $U_1\le S_h\times\{id\}\times \{id\}$, so that $U_1$ is regular. If instead $|C_{k^*}|\ge 2$, then \eqref{gcd2} implies \eqref{gcd}, and so, by ii), $U_2$ is regular. Since $U_1\le U_2$, $U_1$ is regular too.

Assume next that there exists a prime $\pi$ such that $\pi\mid \gcd\left(\gcd(|B_j|)_{j=1}^s, \,|C_{k^*}|!\right),$ and show that $U_1$ is not regular. Note that $\pi\mid |B_j|$ for all $j\in\{1,\ldots,s\}$, and  $\pi \mid |C_{k^*}|!$, that is, $\pi \le |C_{k^*}|$. Thus, the same argument used to conclude the proof of statement ii) shows that  condition a) in Theorem \ref{Regular} fails.
\end{proof}

\begin{theorem}\label{comm}  Let $B=\{B_j\}_{j=1}^s$ be a partition of $H$ and $C=\{C_k\}_{k=1}^t$ be a partition of $N$. Then:
\begin{itemize}
	\item[i)] $\mathcal{F}^{V(B)\times W(C)\times \{id\}}_{\min}\neq \varnothing$ if and only if \eqref{gcd2} holds true;
	\item[ii)]$\mathcal{F}^{V(B)\times W(C)\times \Omega}_{\min}\neq \varnothing$ if and only if \eqref{gcd} holds true.
\end{itemize}
\end{theorem}
\begin{proof} Apply Theorems  \ref{mainmain} and \ref{teonuovo}.
\end{proof}

It is worth noting that $\mathcal{F}^{V(B)\times W(C)\times \{id\}}_{\min}\neq \varnothing$ is equivalent to $\mathcal{F}^{V(B)\times W(C)\times \Omega}_{\min}\neq \varnothing$, provided that $C$ has at least an element which is not a singleton.

We propose now some simple but interesting consequences of Theorem \ref{comm}. Corollary \ref{president} shows that we can always build a neutral and reversal symmetric minimal majority rule that allows all individuals but one to be anonymous. That special type of partial anonymity can be naturally associated with the presence of a president in the committee. Corollary \ref{families} generalizes Theorem 14 in Bubboloni and Gori (2014) to rules also satisfying reversal symmetry.

\begin{corollary}\label{president} Let
$B=\{B_1, B_2\}$ be the partition of $H$ where $B_1=\{1,\dots,h-1\}$ and $B_2=\{h\}.$ Then, for every partition $C$ of $N$, $\mathcal{F}^{V(B)\times W(C)\times \Omega}_{\min}\neq \varnothing$.
\end{corollary}
\begin{proof} Observe that $\gcd(|B_1|,|B_2|)=1$ and apply Theorem \ref{comm}.
\end{proof}

\begin{lemma}\label{old} The following conditions are equivalent:\vspace{-1mm}
\begin{itemize}\item[i)] $G$ is regular;\vspace{-1mm}
\item[ii)] $S_h\times S_n\times \{id\}$ is regular;\vspace{-1mm}
\item[iii)] $\gcd(h,n!)=1.$
\end{itemize}
\end{lemma}
\begin{proof}  i) $\Rightarrow$  ii). This immediately follows from  $S_h\times S_n\times \{id\}\leq G.$

ii) $\Rightarrow$  iii). Let $S_h\times S_n\times \{id\}$ be regular. Given a prime $\pi$ with $\pi\leq n$, we need to show that $\pi\nmid h.$ Consider $\psi\in S_n$ a cycle of length $\pi,$ $\varphi\in S_h$ a cycle of length $h$ and $(\varphi,\psi,id)\in S_h\times S_n\times \{id\}$. Then by Theorem \ref{Regular}, $\pi=|\psi|_\pi\nmid \gcd(T(\varphi))=h.$

iii) $\Rightarrow$  i). Let $\gcd(h,n!)=1.$ Since $G=V(\{H\})\times W(\{N\})\times \Omega$, we immediately get the regularity of $G$ using  Theorem \ref{teonuovo}.
\end{proof}

\begin{corollary}\label{families}  The following statements hold true:
\begin{itemize}
\item[i)] $\mathcal{F}^{S_h\times \{id\}\times \Omega}_{\min}\neq \varnothing$ if and only if $h$ is odd;
\item[ii)] $\mathcal{F}^{\{id\}\times S_n\times \Omega}_{\min}\neq \varnothing$;
\item[iii)] $\mathcal{F}^{S_h\times S_n\times \{id\}}_{\min}\neq \varnothing$ if and only if $\gcd(h,n!)=1$;
\item[iv)]$\mathcal{F}^{G}_{\min}\neq \varnothing$ if and only if $\gcd(h,n!)=1$.
\end{itemize}
\end{corollary}
\begin{proof} i) By Theorem \ref{mainmain}, we have to study the regularity of $U=S_h\times\{id\}\times \Omega$. Note that $U=V(B)\times W(C)\times \Omega$, where
 $B=\{H\}$ and  $C=\{\{i\}:i\in\{1,\dots,n\}\}$. Then, by Theorem \ref{teonuovo}, we have that $U$ is regular if and only if $\gcd (h, 2)=1$, that is, $h$ is odd.

ii) By Theorem \ref{mainmain}, we have to study the regularity of  $U=\{id\}\times S_n\times \Omega$. Note that $U=V(B)\times W(C)\times \Omega$, where $B=\{\{i\}:i\in\{1,\dots,h\}\}$ and $C=\{N\}$. Then, by Theorem \ref{teonuovo}, we get that $U$ is regular .

iii) By Theorem \ref{mainmain}, we have to study the regularity of  $U=S_h\times S_n\times \{id\}$.  By Lemma \ref{old}, we
know that $U$ is regular if and only if $\gcd(h,n!)=1$.

iv) By Theorem \ref{mainmain}, we have to study the regularity of  $G$.  By Lemma \ref{old}, we
know that $G$ is regular if and only if $\gcd(h,n!)=1$.
\end{proof}

\section{Some applications}\label{example}

The algebraic approach we employ has the advantage to provide a method to potentially build and count all the rules
belonging to  $\mathcal{F}^U$ and $\mathcal{F}^U_{\min}$ for all $U\le G$ regular. In this section we show how to perform such computations in two simple cases.

\subsection{Three individuals and three alternatives}\label{3-3}

Let $h=n=3$ so that $H=\{1,2,3\}$, $N=\{1,2,3\}$, $|\mathcal{P}|=6^3$ and $\rho_0=(13).$ In this case, by Lemma \ref{old}, $G=S_3\times S_3\times \Omega$ is not regular because $\gcd(3,3!)\neq 1$. As a consequence, there is no anonymous, neutral and reversal symmetric minimal majority rule. Consider then the partition $B=\{\{1,2\},\{3\}\}$ of $H$ distinguishing individual $3$, who can be thought, for instance, to be the president of the committee, and the  partition $C=\{N\}$ of $N$.  Corollary
\ref{president}, guarantees the existence of $B$-anonymous, neutral
 and reversal symmetric minimal majority rules, that is, $U$-symmetric minimal majority rules, where $U=V(B)\times S_3\times \Omega$. Note that $U$ has order 24 and if $(\varphi,\psi,\rho)\in U$, then $\varphi(3)=3$.  By Propositions \ref{anr} and \ref{n-an-nu}, we can completely describe these rules by finding a system of representatives  of the $U$-orbits $P=(p^j)_{j=1}^{R(U)}$, and determining, for every $j\in\{1,\ldots,R(U)\}$,  the sets $S^U_2(p^j)$.

Since, for every $p\in\mathcal{P}$, the orbit $p^U$ contains a preference profile with preference relation associated with individual 3 given by $id$,
we construct $P$ by selecting  the preference profiles in $\mathcal{P}^{id}=\{p\in \mathcal{P}: p_3=id\}. $ We claim that if $p\in \mathcal{P}^{id}$, then \[
 \mathrm{Stab}_U(p)\leq T=\{(id,id,id),\  ((12),id,id),\ (id,\rho_0,\rho_0),\   ((12),\rho_0,\rho_0)\}.
\]
Namely, since $p_3=id$, if $(\varphi, \psi,\rho)\in \mathrm{Stab}_U(p)$, then we also have  $p^{(\varphi, \psi,\rho)}_3=id$. Moreover, $\psi\rho=id$, so that $\psi=\rho^{-1}=\rho.$ In other words, $\mathrm{Stab}_U(p)\leq \{(\varphi, \rho,\rho):\varphi\in\{(12),id\},\   \rho\in\Omega\}=T$.
It immediately follows that, for every $p\in\mathcal{P}^{id}$, we have
\[
S_1^U(p)
=\left\{
\begin{array}{ll}
\mathcal{L}(N) & \mbox{if }\;\mathrm{Stab}_U(p)\cap \{(id,\rho_0,\rho_0),\   ((12),\rho_0,\rho_0)\}=\varnothing\\
\vspace{-2mm}\\
\Omega          &  \mbox{if }\;\mathrm{Stab}_U(p)\cap \{(id,\rho_0,\rho_0),\   ((12),\rho_0,\rho_0)\}\neq \varnothing,\\
\end{array}\right.
\]
since the centralizer $C_{S_3}(\rho_0)$ is equal to $\Omega$.

Note that $T$ is a group of order $4$ which is generated by  $((12),id,id)$ and $ (id,\rho_0,\rho_0).$
In particular, for every $p\in \mathcal {P}$, we have  that $|\mathrm{Stab}_U(p)|$ divides $4$ and therefore, recalling \eqref{orbord}, we get $\{|p^U|: p\in \mathcal{P}\}\subseteq \{6,12,24\}.$
We split the set $\mathcal{P}^U$ of $U$-orbits into the three disjoint subsets
\[
\mathcal{A}=\{p^U\in \mathcal{P}^U: |p^U|=6\}, \quad \mathcal{B}=\{p^U\in \mathcal{P}^U: |p^U|=12\}, \quad \mathcal{C}=\{p^U\in \mathcal{P}^U: |p^U|=24\},
\]
and put $a=|\mathcal{A}|,\ b=|\mathcal{B}|,\ c=|\mathcal{C}|$. Since the $U$-orbits give a partition of $\mathcal{P}$, we have $6a+12b+24c=6^3$, that is, \begin{equation}\label{class}a+2b+4c=36.\end{equation}

We construct $P$ selecting separately a set of representatives for the orbits in $\mathcal{A}$, $\mathcal{B}$ and $\mathcal{C}.$ Let  $p=(\sigma, \mu\, |\, id)\in \mathcal{P}^{id},$ where $\sigma,\mu\in S_3$. The symbol $|$ is introduced to put in evidence the elements of the partition $B$.
If $p^U\in \mathcal{A}$, then $\mathrm{Stab}_U(p)=T$. From $((12),id,id)\in \mathrm{Stab}_U(p),$ it follows that $\sigma=\mu$, while $ (id,\rho_0,\rho_0)\in \mathrm{Stab}_U(p)$ gives $\rho_0\sigma\rho_0=\sigma$ and therefore $\sigma\in C_{S_3}(\rho_0)=\Omega$. Thus, we have just two choices for $p$ given by $p^1=(id, id\, |\, id)$ and $p^2=(\rho_0, \rho_0\, |\, id).$  Both of them lie effectively in $\mathcal{A}$ and
 $(p^{1})^U\neq (p^{2})^U$ because $p^1$ has three equal components while $p^2$ has not. This proves $a=2$ and gives the first two representatives in $P.$ Next let $p^U\in \mathcal{B}$. Then $|\mathrm{Stab}_U(p)|=2$, so that $\mathrm{Stab}_U(p)$ is generated by one element belonging to $\{ ((12),id,id),\ (id,\rho_0,\rho_0),\   ((12),\rho_0,\rho_0)\}.$ We analyse, by a case by case argument, those three possibilities leaving some easy details to the reader. $\mathrm{Stab}_U(p)=\langle((12),id,id)\rangle$ if and only if $\sigma=\mu,$ with $\sigma\notin \Omega$; the preference profiles $p^3=((132),(132)\, |\, id)$ and $p^4=((23), (23)\, |\, id)$ generate distinct orbits, while
$((123), (123)\, |\, id)\in (p^3)^U$ and $((12), (12)\, |\, id)\in (p^4)^U$.
 $\mathrm{Stab}_U(p)=\langle(id,\rho_0,\rho_0)\rangle$  if and only if $\sigma=id$ and $\mu=\rho_0$ or $\sigma=\rho_0$ and $\mu=id$; these choices lead to preference profiles both in the orbit  of $p^5=(id, \rho_0\, |\, id).$
$\mathrm{Stab}_U(p)=\langle((12),\rho_0,\rho_0)\rangle$ if and only if $\mu=\rho_0\sigma\rho_0$, with $\sigma\not\in\Omega$; this gives the two distinct  orbits $(p^{6})^U$ and $(p^{7})^U$, where $p^6=((123), (132)\, |\, id)$ and  $p^7=((12), (23)\, |\, id)$.
 It is easily checked that no coincidence is possible among the orbits $(p^j)^U$ for $j\in\{3,\dots, 7\},$ and thus $b=5.$ By relation \eqref{class}, we then have $c=6$ and  $R(U)=13$. As a consequence, we are left with finding $6$ preference profiles whose orbits are distinct and in $\mathcal{C}$. It can be easily proved that the desired representatives are $p^8,\ldots, p^{13}$ in the list below, where we explicitly write in the matrix form all the representatives for the $U$-orbits:
\[
p^1=
\left[
\begin{array}{cc|c}
 1 & 1 &  1 \\
 2 & 2 &  2\\
 3 & 3 &  3\\
\end{array}
\right]
,\quad
p^2=
\left[
\begin{array}{cc|c}
 3 & 3 & 1  \\
 2 & 2 & 2\\
 1 & 1 & 3\\
\end{array}
\right],\quad
p^3=
\left[
\begin{array}{cc|c}
 3 & 3 & 1\\
 1 & 1 & 2 \\
 2 & 2 & 3\\
\end{array}
\right],
\quad
p^4=
\left[
\begin{array}{cc|c}
 1 & 1 & 1 \\
 3 & 3 &  2\\
 2 & 2 &  3\\
\end{array}
\right],
\]
\[
p^5=
\left[
\begin{array}{cc|c}
 1 & 3 & 1   \\
 2 & 2 & 2  \\
 3 & 1 &  3\\
\end{array}
\right],\quad
p^6=
\left[
\begin{array}{cc|c}
 2 & 3 & 1  \\
 3 & 1 &  2 \\
 1 & 2 &  3 \\
\end{array}
\right]
,\quad
p^{7}=
\left[
\begin{array}{cc|c}
 2 & 1 & 1  \\
 1 & 3 &  2 \\
 3 & 2 &  3 \\
\end{array}
\right],\quad
p^{8}=
\left[
\begin{array}{cc|c}
 1 & 2 & 1 \\
 2 & 1 & 2\\
 3 & 3 & 3 \\
\end{array}
\right],
\]
\[
\quad
p^{9}=
\left[
\begin{array}{cc|c}
 1 & 2 & 1  \\
 2 & 3 & 2 \\
 3 & 1 & 3 \\
\end{array}
\right],
\quad
p^{10}=
\left[
\begin{array}{cc|c}
 3 & 3 & 1 \\
 2 & 1 &  2 \\
 1 & 2 &  3 \\
\end{array}
\right],\quad
p^{11}=
\left[
\begin{array}{cc|c}
 3& 1 & 1 \\
 2 & 3& 2  \\
 1& 2 & 3 \\
\end{array}
\right],\quad
p^{12}=
\left[
\begin{array}{cc|c}
 2 & 1 & 1 \\
 3& 3 & 2  \\
 1 & 2 & 3  \\
\end{array}
\right],
\]
\[
p^{13}=
\left[
\begin{array}{cc|c}
 2 & 2 & 1  \\
 3 & 1 & 2  \\
 1 & 3 & 3  \\
\end{array}
\right].
\]
Little further work allows to get the following table:

\begin{equation*}
\begin{array}{|c||c|c|c|c|c|c|}
\hline
       & C_{2} & C_{3} &   S_1^U  & S_2^U  \\ \hline\hline
p^{1} & \left\{[1,2,3]^T\right\} & \left\{[1,2,3]^T\right\} & \left\{[1,2,3]^T,  [3,2,1]^T\right\}    &  \{[1,2,3]^T\}    \\ \hline
p^{2} & \left\{[3,2,1]^T\right\} &  \mathcal{L}(\{1,2,3\})    & \left\{[1,2,3]^T,  [3,2,1]^T\right\}   & \left\{[3,2,1]^T\right\}      \\ \hline
p^{3} & \left\{[3,1,2]^T\right\} & \left\{[1,2,3]^T, [1,3,2]^T, [3,1,2]^T\right\} &  \mathcal{L}(\{1,2,3\})   &  \{[3,1,2]^T\}        \\ \hline
p^{4} & \left\{[1,3,2]^T\right\} & \left\{[1,2,3]^T\, [1,3,2]^T \right\} &  \mathcal{L}(\{1,2,3\})    & \left\{[1,3,2]^T\right\}      \\ \hline
p^{5} & \left\{[1,2,3]^T\right\} &  \mathcal{L}(\{1,2,3\})   &  \left\{[1,2,3]^T,  [3,2,1]^T\right\}   &  \left\{[1,2,3]^T\right\}       \\ \hline
p^{6} & \varnothing & \mathcal{L}(\{1,2,3\}) &  \left\{[1,2,3]^T,  [3,2,1]^T\right\}     & \left\{[1,2,3]^T,  [3,2,1]^T\right\}       \\ \hline
p^{7} & \{[1,2,3]^T\} & \{[1,2,3]^T,[2,1,3]^T,[1,3,2]^T\} & \left\{[1,2,3]^T,  [3,2,1]^T\right\}   &    \{[1,2,3]^T\}     \\ \hline
p^{8} & \{[1,2,3]^T\} & \{[1,2,3]^T,[2,1,3]^T\} & \mathcal{L}(\{1,2,3\})   &   \{[1,2,3]^T\} \\ \hline
p^{9} & \{[1,2,3]^T\} & \{[1,2,3]^T,[2,1,3]^T, [3,2,1]^T\} & \mathcal{L}(\{1,2,3\})  &  \{[1,2,3]^T\}          \\ \hline
p^{10} & \{[3,1,2]^T\} &\mathcal{L}(\{1,2,3\}) & \mathcal{L}(\{1,2,3\})   &    \{[3,1,2]^T\}   \\ \hline
p^{11} & \{[1,3,2]^T\} & \mathcal{L}(\{1,2,3\})  & \mathcal{L}(\{1,2,3\})   &  \{[1,3,2]^T\}   \\ \hline
p^{12} & \{[1,2,3]^T\} & \{[1,2,3]^T,[2,1,3]^T,[2,3,1]^T\} & \mathcal{L}(\{1,2,3\})   &  \{[1,2,3]^T\}   \\ \hline
p^{13} & \{[2,1,3]^T\} &  \{[1,2,3]^T,[2,1,3]^T,[2,3,1]^T\} & \mathcal{L}(\{1,2,3\})   & \{[2,1,3]^T\}   \\ \hline
\end{array}
\end{equation*}
Looking at the table and using   Propositions \ref{anr} and \ref{n-an-nu}, we deduce that $|\mathcal{F}^{U}|=2^{13}3^{8}$ and $|\mathcal{F}^{U}_{\min}|=2.$
In particular, there are only two $U$-symmetric minimal majority rules one for each possible choice related to $p^6$.

We finally observe that the rules in  $\mathcal{F}^{U}_{\min}$ can be effectively described in terms of the function $S$ defined in \eqref{smr}. Indeed, we have that $\mathcal{F}_{\min}^U=\{F_1,F_2\}$, where $F_1$ and $F_2$ are defined, for every $p\in\mathcal{P}$, as
\[
F_1(p)=
\left\{
\begin{array}{ll}
S(p) &\mbox{if }S(p)\in\mathcal{L}(N)\\
\vspace{-2mm}\\
p_3 & \mbox{if }S(p)\not \in\mathcal{L}(N)
\end{array}
\right.\quad \mbox{ and }\quad
F_2(p)=
\left\{
\begin{array}{ll}
S(p) &\mbox{if }S(p)\in\mathcal{L}(N)\\
\vspace{-2mm}\\
p_3\rho_0 & \mbox{if }S(p)\not \in\mathcal{L}(N).
\end{array}
\right.
\]
Thus, having found an agreement on the principles of minimal majority, neutrality, reversal symmetry and anonymity with respect to the subcommittees $\{1,2\}$ and $\{3\}$, we have that committee members are left with deciding which rule to employ between $F_1$ and $F_2$. Of course, if individual $3$ is assumed to have more decision power than other individuals, then the choice of $F_1$ is definitely more appropriate.

\subsection{Five individuals and three alternatives}

Let $h=5$  and $n=3$ so that, as in the previous example, $\rho_0=(13)$. Since $\gcd(5,3!)=1$,
by Corollary \ref{families}, there exists an anonymous, neutral and reversal symmetric minimal majority rule, that is, a $G$-symmetric minimal majority rule where $G=S_5\times S_3\times \Omega$.
In order to apply  Propositions \ref{anr} and \ref{n-an-nu}, we need a system of representatives of the $G$-orbits.
Bubboloni and Gori (2014, Section 10) consider $U=S_5\times S_3\times \{id\}$,
compute $R(U)=42,$ and construct a system of representatives of the $U$-orbits
 $\hat{P}=(\hat{p}^i)_{i=1}^{42}$. Here we extract from $\hat{P}$ a system of representatives of the $G$-orbits $P=(p^j)_{j=1}^{R(G)}$.

In order to determine  $P$, we
scroll the list $\hat{P}$ starting from the beginning, inquiring if a certain $\hat{p}^i$ has a stabilizer containing an element of the type $(\varphi, \psi,(13))\in G,$ with $\psi$ a conjugate of $(13)$,  that is, $\psi\in\{(12), (13),(23)\}.$
If that happens, then $[\mathrm{Stab}_G(\hat{p}^i):  \mathrm{Stab}_U(\hat{p}^i)]=2$,
$(\hat{p}^{i})^G=(\hat{p}^{i})^U$ and $S_1^G(\hat{p}^i)=\{q\in \mathcal{L}(\{1,2,3\}) : \psi q (13)=q\}$; in this case we put $\hat{p}^i$ in the list  $P$.
 If that does not happen, then there exists a unique $k\in\mathbb{N}$ with  $i< k\leq 42$  such that $(\hat{p}^{i})^G=(\hat{p}^{i})^U\cup (\hat{p}^{k})^U$ and we  have $S_1^G(\hat{p}^i)=\mathcal{L}(\{1,2,3\})$; in this case we  put $\hat{p}^i$ in the list $P$ and eliminate $\hat{p}^{k}$ from the list $\hat{P}$.\footnote{All the facts above are  consequence of $[G:U]=2.$ Further details  about this example can be found in Bubboloni and Gori (2013). }

At the end of the described procedure we get the following list of representatives for the $G$-orbits, expressed in the matrix form:
\[
p^1=
\left[
\begin{array}{ccccc}
 1 & 1 &  1&  1& 1 \\
 2 & 2 &  2&  2& 2\\
 3 & 3 &  3&  3& 3 \\
\end{array}
\right],\quad
p^2=
\left[
\begin{array}{ccccc}
 1 & 1 & 1 & 1 & 2 \\
 2 & 2 & 2&  2&  1\\
 3 & 3 & 3 & 3 & 3 \\
\end{array}
\right],\quad
p^3=
\left[
\begin{array}{ccccc}
 1 & 1 & 1 & 1 & 3 \\
 2 & 2 & 2 & 2 & 2 \\
 3 & 3 & 3 & 3 & 1 \\
\end{array}
\right],
\]
\[
p^4=
\left[
\begin{array}{ccccc}
 1 & 1 & 1 & 1 & 2 \\
 2 & 2 &  2& 2 & 3\\
 3 & 3 &  3& 3 & 1\\
\end{array}
\right],\quad
p^5=
\left[
\begin{array}{ccccc}
 1 & 1 & 1 & 2 &2  \\
 2 & 2 & 2 & 1 & 1 \\
 3 & 3 &  3&  3&  3\\
\end{array}
\right],\quad
p^6=
\left[
\begin{array}{ccccc}
 1 & 1 & 1 & 3 & 3 \\
 2 & 2 &  2& 2 & 2 \\
 3 & 3 &  3& 1 & 1 \\
\end{array}
\right]
\]
\[
p^{7}=
\left[
\begin{array}{ccccc}
 1 & 1 & 1 & 2 & 2 \\
 2 & 2 &  2& 3 & 3 \\
 3 & 3 &  3& 1 & 1 \\
\end{array}
\right],\quad
p^{8}=
\left[
\begin{array}{ccccc}
 1 & 1 & 1 & 2 & 3 \\
 2 & 2 & 2 & 1 & 2\\
 3 & 3 & 3 & 3 & 1\\
\end{array}
\right],\quad
p^{9}=
\left[
\begin{array}{ccccc}
 1 & 1 & 1 & 2 & 1 \\
 2 & 2 & 2 &  1&  3\\
 3 & 3 & 3 &  3&  2\\
\end{array}
\right],
\]
\[
p^{10}=
\left[
\begin{array}{ccccc}
 1 & 1 & 1 & 2 & 2 \\
 2 & 2 &  2& 1 & 3 \\
 3 & 3 &  3& 3 & 1 \\
\end{array}
\right],\quad
p^{11}=
\left[
\begin{array}{ccccc}
 1 & 1 & 1 & 2 & 3 \\
 2 & 2 & 2 & 1 & 1 \\
 3 & 3 & 3 & 3 & 2 \\
\end{array}
\right],\quad
p^{12}=
\left[
\begin{array}{ccccc}
 1 & 1 & 1 & 3 & 2 \\
 2 & 2 & 2 & 2 & 3 \\
 3 & 3 & 3 & 1 & 1 \\
\end{array}
\right],
\]
\[
p^{13}=
\left[
\begin{array}{ccccc}
 1 & 1 & 1 & 2 & 3 \\
 2 & 2 & 2 & 3 & 1 \\
 3 & 3 & 3 & 1 & 2 \\
\end{array}
\right],\quad
p^{14}=
\left[
\begin{array}{ccccc}
 1 & 1 &  2& 2 & 3 \\
 2 & 2 &  1& 1 & 2 \\
 3 & 3 &  3& 3 & 1 \\
\end{array}
\right],\quad
p^{15}=
\left[
\begin{array}{ccccc}
 1 & 1 & 2 & 2 & 1 \\
 2 & 2 & 1 & 1 &  3\\
 3 & 3 & 3 & 3 &  2\\
\end{array}
\right],
\]
\[
p^{16}=
\left[
\begin{array}{ccccc}
 1 & 1 & 3 & 3 & 2 \\
 2 & 2 & 2 & 2 & 1 \\
 3 & 3 & 1 & 1 & 3 \\
\end{array}
\right],\quad
p^{17}=
\left[
\begin{array}{ccccc}
 1 & 1 & 2 & 2 &  2\\
 2 & 2 & 3 & 3 &  1\\
 3 & 3 & 1 & 1 &  3\\
\end{array}
\right],\quad
p^{18}=
\left[
\begin{array}{ccccc}
 1 & 1 & 2 & 2 & 3 \\
 2 & 2 & 3 & 3 & 2 \\
 3 & 3 & 1 & 1 & 1 \\
\end{array}
\right],
\]
\[
p^{19}=
\left[
\begin{array}{ccccc}
 1 & 1 & 3 & 3 & 2 \\
 2 & 2 & 1 & 1 & 3 \\
 3 & 3 & 2 & 2 & 1 \\
\end{array}
\right],\quad
p^{20}=
\left[
\begin{array}{ccccc}
 1 & 1 & 2 & 3 & 1 \\
 2 & 2 & 1 & 2 & 3 \\
 3 & 3 & 3 & 1 & 2 \\
\end{array}
\right],\quad
p^{21}=
\left[
\begin{array}{ccccc}
 1 & 1 & 2 & 3 & 2 \\
 2 & 2 & 1 & 2 & 3 \\
 3 & 3 & 3 & 1 & 1 \\
\end{array}
\right],
\]
\[
p^{22}=
\left[
\begin{array}{ccccc}
 1 & 1 & 2 & 3 &  3\\
 2 & 2 & 1 & 2 &  1\\
 3 & 3 & 3 & 1 &  2\\
\end{array}
\right],\quad
p^{23}=
\left[
\begin{array}{ccccc}
 1 & 1 & 2 & 1 & 2 \\
 2 & 2 & 1 & 3 & 3 \\
 3 & 3 & 3 & 2 & 1 \\
\end{array}
\right],\quad
p^{24}=
\left[
\begin{array}{ccccc}
 1 & 1 & 2 & 2 & 3 \\
 2 & 2 & 1 & 3 & 1 \\
 3 & 3 & 3 & 1 & 2 \\
\end{array}
\right],
\]
\[
p^{25}=
\left[
\begin{array}{ccccc}
 1 & 1 & 3 & 2 & 3 \\
 2 & 2 & 2 & 3 & 1 \\
 3 & 3 & 1 & 1 & 2 \\
\end{array}
\right],\quad
p^{26}=
\left[
\begin{array}{ccccc}
 1 & 2 & 3 & 1 & 2 \\
 2 & 1 & 2 & 3 & 3 \\
 3 & 3 & 1 & 2 & 1 \\
\end{array}
\right].
\]
In particular, we obtain $R(G)=26$. Looking at the representatives, a simple but tedious computation leads to the following table:

\begin{sideways}
\begin{minipage}{\textheight}
\begin{equation*}
\begin{array}{|c||c|c|c|c|c|c|c|}
\hline
       & C_{3} & C_{4} & C_{5} &  S_1^G  & S_2^G  \\ \hline\hline
p^{1} & \left\{[1,2,3]^T\right\} & \left\{[1,2,3]^T\right\} & \left\{[1,2,3]^T\right\}   &  \{[1,2,3]^T, [3,2,1]^T\}   &  \{[1,2,3]^T\}   \\ \hline
p^{2} & \left\{[1,2,3]^T\right\} & \left\{[1,2,3]^T\right\} & \left\{[1,2,3]^T, [2,1,3]^T\right\}    &  \mathcal{L}(\{1,2,3\})  &  \{[1,2,3]^T\}     \\ \hline
p^{3} & \left\{[1,2,3]^T\right\} & \left\{[1,2,3]^T\right\} &  \mathcal{L}(\{1,2,3\})   &  \{[1,2,3]^T,[3,2,1]^T\}   &  \{[1,2,3]^T\}       \\ \hline
p^{4} & \left\{[1,2,3]^T\right\} & \left\{[1,2,3]^T\right\} & \left\{[1,2,3]^T,[2,1,3]^T,[2,3,1]^T\right\}    &   \mathcal{L}(\{1,2,3\})   &  \{[1,2,3]^T\}   \\ \hline
p^{5} & \left\{[1,2,3]^T\right\} & \{[1,2,3]^T,[2,1,3]^T\} & \{[1,2,3]^T,[2,1,3]^T\}   &   \mathcal{L}(\{1,2,3\})   &  \{[1,2,3]^T\}   \\ \hline
p^{6} & \left\{[1,2,3]^T\right\} & \mathcal{L}(\{1,2,3\}) &  \mathcal{L}(\{1,2,3\})   &   \{[1,2,3]^T,[3,2,1]^T\}  &  \{[1,2,3]^T\}      \\ \hline
p^{7} & \{[1,2,3]^T\} & \{[1,2,3]^T,[2,1,3]^T,[2,3,1]^T\} & \{[1,2,3]^T,[2,1,3]^T,[2,3,1]^T\}   &  \mathcal{L}(\{1,2,3\})   &  \{[1,2,3]^T\}        \\ \hline
p^{8} & \{[1,2,3]^T\} & \{[1,2,3]^T,[2,1,3]^T\} & \mathcal{L}(\{1,2,3\})   &  \mathcal{L}(\{1,2,3\})   &  \{[1,2,3]^T\} \\ \hline
p^{9} & \{[1,2,3]^T\} & \{[1,2,3]^T\} & \{[1,2,3]^T,[3,2,1]^T,[1,3,2]^T\}   &  \{[1,2,3]^T,[3,2,1]^T\}   &  \{[1,2,3]^T\}        \\ \hline
p^{10} & \{[1,2,3]^T\} & \{[1,2,3]^T,[2,1,3]^T\} & \{[1,2,3]^T,[2,1,3]^T,[2,3,1]^T\}   &    \mathcal{L}(\{1,2,3\})   &  \{[1,2,3]^T\}   \\ \hline
p^{11} & \{[1,2,3]^T\} & \{[1,2,3]^T\} & \mathcal{L}(\{1,2,3\})   &   \mathcal{L}(\{1,2,3\})    &  \{[1,2,3]^T\}   \\ \hline
p^{12} & \{[1,2,3]^T\} & \{[1,2,3]^T,[2,1,3]^T,[2,3,1]^T\} & \mathcal{L}(\{1,2,3\})   &    \mathcal{L}(\{1,2,3\})   &  \{[1,2,3]^T\} \\ \hline
p^{13} & \{[1,2,3]^T\} & \{[1,2,3]^T,[2,3,1]^T\} & \mathcal{L}(\{1,2,3\})   &  \{[1,2,3]^T,[3,2,1]^T\}  &  \{[1,2,3]^T\}     \\ \hline
p^{14} & \{[2,1,3]^T\} & \{[1,2,3]^T,[2,1,3]^T\} & \mathcal{L}(\{1,2,3\})   &  \mathcal{L}(\{1,2,3\})    &  \{[2,1,3]^T\}     \\ \hline
p^{15} & \{[1,2,3]^T\} & \{[1,2,3]^T,[2,1,3]^T\} & \{[1,2,3]^T,[3,2,1]^T,[1,3,2]^T\}   &   \mathcal{L}(\{1,2,3\})  &  \{[1,2,3]^T\}   \\ \hline
p^{16} & \{[2,1,3]^T\} & \mathcal{L}(\{1,2,3\}) & \mathcal{L}(\{1,2,3\})   &  \mathcal{L}(\{1,2,3\})    &  \{[2,1,3]^T\}     \\ \hline
p^{17} & \{[2,1,3]^T\} & \{[1,2,3]^T,[2,1,3]^T,[2,3,1]^T\} & \{[1,2,3]^T,[2,1,3]^T,[2,3,1]^T\}   &  \{[2,1,3]^T,[3,1,2]^T\}   &  \{[2,1,3]^T\}  \\ \hline
p^{18} & \{[2,3,1]^T\} & \{[1,2,3]^T,[2,1,3]^T,[2,3,1]^T\} & \mathcal{L}(\{1,2,3\})   &  \mathcal{L}(\{1,2,3\})   &  \{[2,3,1]^T\}    \\ \hline
p^{19} & \varnothing  & \{[1,2,3]^T,[1,3,2]^T,[3,1,2]^T\} & \mathcal{L}(\{1,2,3\})   &  \{[1,3,2]^T,[2,3,1]^T\}    &  \{[1,3,2]^T\}     \\ \hline
p^{20} & \{[1,2,3]^T\} & \{[1,2,3]^T,[3,2,1]^T,[1,3,2]^T\} & \mathcal{L}(\{1,2,3\})   &   \{[1,2,3]^T,[3,2,1]^T\}  &  \{[1,2,3]^T\}       \\ \hline
p^{21} & \{[2,1,3]^T\} & \{[1,2,3]^T,[2,1,3]^T,[2,3,1]^T\} & \mathcal{L}(\{1,2,3\})   &  \mathcal{L}(\{1,2,3\})  &  \{[2,1,3]^T\}    \\ \hline
p^{22} & \{[1,2,3]^T\} & \mathcal{L}(\{1,2,3\}) & \mathcal{L}(\{1,2,3\})   &  \mathcal{L}(\{1,2,3\})   &  \{[1,2,3]^T\}    \\ \hline
p^{23} & \{[1,2,3]^T\} & \{[1,2,3]^T,[2,1,3]^T\} & \mathcal{L}(\{1,2,3\})   &   \mathcal{L}(\{1,2,3\})  &  \{[1,2,3]^T\}      \\ \hline
p^{24} & \{[1,2,3]^T\} & \mathcal{L}(\{1,2,3\}) & \mathcal{L}(\{1,2,3\})   &  \mathcal{L}(\{1,2,3\})  &  \{[1,2,3]^T\}    \\ \hline
p^{25} & \varnothing  & \mathcal{L}(\{1,2,3\}) & \mathcal{L}(\{1,2,3\})   &   \{[1,2,3]^T,[3,2,1]^T\}   &    \{[1,2,3]^T,[3,2,1]^T\}       \\ \hline
p^{26} & \{[2,1,3]^T\} & \mathcal{L}(\{1,2,3\}) & \mathcal{L}(\{1,2,3\})   &   \{[2,1,3]^T,[3,1,2]^T\}    &  \{[2,1,3]^T\}   \\ \hline
\end{array}
\end{equation*}
\end{minipage}
\end{sideways}

Looking at the table and using  Propositions \ref{anr} and \ref{n-an-nu}, we deduce in particular that $|\mathcal{F}^{G}|=2^{26}3^{16}$ and $|\mathcal{F}^{G}_{\min}|=2$.
Then, there are only two possible anonymous, neutral and reversal symmetric minimal majority rules, one for each possible social outcome associated with $p^{25}.$ As $[1,2,3]^T$ is present in $p^{25}$ twice, while  $[3,2,1]^T$ only once, the choice of $[1,2,3]^T$ could be judged more appropriate.

We finally observe that $\mathcal{F}^U_{\min}\setminus \mathcal{F}^G_{\min}\neq \varnothing$ because, as proved in Bubboloni and Gori (2014, Section 10), $|\mathcal{F}^U_{\min}|=18$. That shows, as expected, that there are anonymous and neutral minimal majority rules that are not reversal symmetric.

\section{Proof of Theorem \ref{mainmain}}\label{appendix}

Assume that $\mathcal{F}^U_{\min}\neq \varnothing$. Then $\mathcal{F}^U\neq \varnothing$ and, by Theorem \ref{main}, $U$ is regular.
In order to prove the converse, assume that $U$ is regular. By Theorem \ref{main}, we know that $\mathcal{F}^U\neq \varnothing$ and then, for every $p\in\mathcal{P}$, $S^U_1(p)\neq\varnothing$. We are going to prove that, for every $p\in\mathcal{P}$, $S^U_2(p)\neq\varnothing$, as well. Indeed, by Proposition  \ref{n-an-nu}, that implies $\mathcal{F}^U_{\min}\neq \varnothing$. Of course, for every $p\in\mathcal{P}$ such that $\mathrm{Stab}_U(p)\le S_h \times \{id\}\times \{id\}$, we have $S^U_1(p)=\mathcal{L}(N)$ so that
$S^U_2(p)=C_{\nu(p)}(p)\neq\varnothing$. Thus, we are left with proving that, for every $p\in\mathcal{P}$ such that there exists  $(\varphi,\psi,\rho_0)\in \mathrm{Stab}_U(p)$ with $\psi$ conjugate
to $\rho_0$, we have $S^U_2(p)\neq\varnothing$.

\vspace{2mm}

\noindent {\it From now till the end of the section, let us fix  $p\in\mathcal{P}$ and $(\varphi,\psi,\rho_0)\in \mathrm{Stab}_U(p)$ with $\psi$ conjugate to $\rho_0.$ For simplicity of notation, we set $\nu =\nu(p)$. Recall that, under the assumption that $U$ is regular, we have to prove that $S^U_2(p)=S^U_1(p)\cap C_{\nu}(p) \neq\varnothing$ . }

\vspace{2mm}

We are going to exhibit an element of the set $S^U_2(p)$, namely the linear order $q$ defined in \eqref{q-magic}. The construction of $q$ is quite tricky and relies on some preliminary lemmas concerning the properties of the relations $\Sigma_{\nu}(p)$ and  $\Sigma^C_{\nu}(p)$ defined in \eqref{sigmanu} and \eqref{sigmac}. Thus, the first part of the section is devoted to the study of such relations.

First of all, being $\psi$  a conjugate of $\rho_0$,  we have that $\psi$ has the same type of $\rho_0$ and, in particular, $|\psi|=2$. Let $k=r(\psi)$, $K=\{1,\ldots,k\}$, and let $(\hat{x}_j)_{j=1}^k\in N^k$ be an ordered system of representatives of the $\psi$-orbits. Then, we have
\[
O(\psi)=\left\{
\{\hat{x}_j,\psi(\hat{x}_j)\}: j\in K
\right\}.
\]
Note that if $n$ is even, then $k=\frac{n}{2}$; $|\{\hat{x}_j,\psi(\hat{x}_j)\}|=2$ for all $j\in K$; $\psi$ has no fixed point. If instead $n$ is odd, then $k=\frac{n+1}{2}$; $|\{\hat{x}_j,\psi(\hat{x}_j)\}|=2$ for all $j\in K\setminus\{k\}$; $\hat{x}_k$ is the unique fixed point of $\psi$.

Consider now the relation on $N$ given by
\begin{equation}\label{sigmanu}
\Sigma_{\nu}(p)=\left\{(x,y)\in N\times N: |\{i\in H: x>_{p_i} y\}|\ge \nu\right\}.
\end{equation}
So $(x,y)\in\Sigma_{\nu}(p)$ means that at least $\nu$ individuals prefer $x$ to $y$ with respect to the preference profile $p$. Observe that
$C_{\nu}(p)=\{f\in\mathcal{L}(N): f\supseteq \Sigma_{\nu}(p)\}$ and since $C_{\nu}(p)$ is non-empty, we have that $\Sigma_{\nu}(p)$ is acyclic. Note also that $\Sigma_{\nu}(p)$ is asymmetric and generally not transitive and not complete. Thus, $x\geq _{\Sigma_{\nu}(p)} y$ is equivalent to $x> _{\Sigma_{\nu}(p)} y$ and implies $x\neq y.$
We will write, for compactness, $x>_{\nu}y$ instead of $x> _{\Sigma_{\nu}(p)} y$.
Our first result is about the role of $\psi$ in the relation $\Sigma_{\nu}(p)$.
\begin{lemma}\label{nupsi}
Let $x,y\in N$. Then  $x>_{\nu} y$  if and only if $\psi(y)>_{\nu} \psi(x)$.
\end{lemma}

\begin{proof} Let $x,y\in N$ and consider the two subsets of $H$ given by $A=\{i\in H: x>_{p_i} y\}$ and $B=\{i\in H: \psi(y)>_{p_i} \psi(x)\}$. Clearly it is enough to show  $|A|=|B|$. We do that proving  that $\varphi(A)\subseteq B$ and $\varphi(B)\subseteq A.$  If $i\in A$, then we have $x>_{p_i} y$ and thus, by the properties \eqref{psiR} and \eqref{ro}, we get $\psi(y)>_{\psi p_i\rho_0}\psi(x).$ But since $(\varphi,\psi,\rho_0)\in \mathrm{Stab}_U(p)$, we have that $\psi p_i\rho_0=p_{\varphi(i)}$ and thus $\psi(y)>_{ p_{\varphi(i)}}\psi(x),$ that is, $\varphi(i)\in B.$ Next let $i\in B$, that is, $\psi(y)>_{p_i} \psi(x).$ Arguing as before, we get $x=\psi\psi(x)>_{p_{\varphi(i)}} \psi\psi(y)=y,$ which means $\varphi(i)\in A.$
\end{proof}

Given $x,y\in N$ with $x\neq y$, a chain $\gamma$ for $\Sigma_{\nu}(p)$ (or a $\Sigma_{\nu}(p)$-chain) from $x$ to $y$ is an ordered sequence  $x_1,\ldots, x_l,$ with $l\ge 2,$ of distinct elements of $N$ such that $x_1=x$, $x_l=y,$ and, for every $j\in\{1,\ldots,l-1\}$, $(x_j,x_{j+1})\in\Sigma_{\nu}(p)$. The number $l-1$ is called the length of the chain, $x$ its starting point and $y$ its end point.

Consider then the following relation on $N$,
\begin{equation}\label{sigmac}
\Sigma^C_{\nu}(p)=\{(x,y)\in N\times N:
\mbox{there exists a}\  \Sigma_{\nu}(p)\mbox{-chain from } x\   \mbox{to}\  y\},
\end{equation}
and note that $\Sigma^C_{\nu}(p)\supseteq \Sigma_{\nu}(p).$

\begin{lemma}\label{asymmetry}
$\Sigma^C_{\nu}(p)$ is asymmetric and transitive. Moreover, for every $x,y\in N$, $(x,y)\in \Sigma^C_{\nu}(p)$ if and only if $(\psi(y),\psi(x))\in \Sigma^C_{\nu}(p)$.
\end{lemma}

\begin{proof}
Let us prove first that $\Sigma^C_{\nu}(p)$ is asymmetric.
Let $(x,y)\in \Sigma^C_{\nu}(p)$. Then, there exists a $\Sigma_{\nu}(p)$-chain from $ x$  to $ y$, that is, there exist $l\ge 2$ distinct $x_1,\ldots,x_l\in N$ such that $x=x_1$, $y=x_l$ and, for every $j\in\{1,\ldots,l-1\}$, $x_j>_{\nu}x_{j+1}$.
Assume, by contradiction, that  $(y,x)\in \Sigma^C_{\nu}(p).$ Then there exist $m\ge 2$ distinct $y_1,\ldots,y_m\in N$ such that $y=y_1$, $x=y_m$ and, for every $j\in\{1,\ldots,m-1\}$, $y_j>_{\nu}y_{j+1}$. Consider now the set
$
A=\{j\in\{2,\ldots,m\}: y_j\in\{x_1,\ldots,x_{l-1}\}\}.
$
Clearly, because $y_m=x_1,$ we have $m\in A\neq \varnothing$. Let us define then $m^*=\min A$, so that there exists $l^*\in\{1,\ldots,l-1\}$ such that $y_{m^*}=x_{l^*}$. Then it is easy to check that $x_{l^*},x_{l^*+1},\ldots, x_l,y_2,\ldots, y_{m^*}$ is a sequence of at least three elements in $N$, with no repetition up to the $x_{l^*}=y_{m^*},$ which is a cycle in $\Sigma_{\nu}(p)$ and the contradiction is found.

Let us prove now that $\Sigma^C_{\nu}(p)$ is transitive. Let $(x,y),(y,z)\in \Sigma^C_{\nu}(p)$.
Then, by the definition of $\Sigma^C_{\nu}(p)$,  there exist $l\ge 2$ distinct $x_1,\ldots,x_l\in N$ such that $x=x_1$, $y=x_l$ and, for every $j\in\{1,\ldots,l-1\}$, $x_j>_{\nu}x_{j+1}$; moreover, there are $m\ge 2$ distinct $y_1,\ldots,y_m\in N$ such that $y=y_1$, $z=y_m$ and, for every $j\in\{1,\ldots,m-1\}$, $y_j>_{\nu}y_{j+1}$. Consider then the sequence $\gamma$ of alternatives  $x_1,\ldots,x_l,y_2,\ldots,y_m$. We show that those alternatives are all distinct. Assume that there exist $i\in \{1,\ldots,l\}$ and $j\in \{1,\ldots,m\}$ with $x_i=y_j$. Then we have a $\Sigma_{\nu}(p)$-chain with starting point $x_i$ and end point $y$ as well as a $\Sigma_{\nu}(p)$-chain with starting point $y$ and end point $y_j=x_i$, that is, $(x_i,y)\in \Sigma^C_{\nu}(p)$ and $(y,x_i)\in \Sigma^C_{\nu}(p),$ against the asymmetry.
It follows that $\gamma$ is a chain from $x$ to $z$, so that $(x,z)\in \Sigma^C_{\nu}(p)$.

We are left with proving that $(x,y)\in \Sigma^C_{\nu}(p)$ if and only if $(\psi(y),\psi(x))\in \Sigma^C_{\nu}(p)$. Let $(x,y)\in \Sigma^C_{\nu}(p)$ and consider $l\ge 2$ distinct $x_1,\ldots,x_l\in N$ such that $x=x_1$, $y=x_l$ and, for every $j\in\{1,\ldots,l-1\}$, $x_j>_{\nu}x_{j+1}$. Defining, for every $j\in\{1,\ldots,l\}$, $y_j=\psi(x_{l-j+1})$ and using Lemma \ref{nupsi}, it is immediately checked that $\psi(y)=y_1$, $\psi(x)=y_l$ and that, for every $j\in\{1,\ldots,l-1\}$, $y_{j}>_{\nu}y_{j+1}$. In other words, we have a $\Sigma^C_{\nu}(p)$-chain from $\psi(y)$ to $\psi(x),$ that is,
 $(\psi(y),\psi(x))\in \Sigma^C_{\nu}(p)$. The other implication is now a trivial consequence of $|\psi|=2$.
\end{proof}

In what follows, we write $x\hookrightarrow_{\nu}y$ instead of $(x,y)\in \Sigma^C_{\nu}(p)$ and $x\not\hookrightarrow_{\nu}y$ instead of $(x,y)\notin \Sigma^C_{\nu}(p)$.
As consequence of Lemma \ref{asymmetry}, for every $x,y,z\in N$, the following relations hold true:
$x\not\hookrightarrow_{\nu}x$;  $x\hookrightarrow_{\nu}y$ implies  $y\not\hookrightarrow_{\nu}x$; $x\hookrightarrow_{\nu}y$ and $y\hookrightarrow_{\nu}z$ imply  $x\hookrightarrow_{\nu}z$; $x\hookrightarrow_{\nu}y$ is equivalent to $\psi(y)\hookrightarrow_{\nu}\psi(x)$.

Define now, for every $z\in N$, the subset of $N$
\[
\Gamma(z)=\{x\in N: x\hookrightarrow_{\nu} z\}.
\]
Note that $z\notin \Gamma(z)$ and that it may happen that $\Gamma(z)=\varnothing$. This is the case exactly when, for every $x\in N$, the relation $x>_\nu z$ does not hold.
Define now the subset of $N$
\[
\Gamma=\mbox{$\bigcup_{z\in N}$}[\Gamma(z)\cap \Gamma(\psi(z))].
\]
This set represents those alternatives $x\in N$ from which we can reach, by a $\Sigma_{\nu}(p)$-chain, both $z$ and $\psi(z),$ for some $z\in N.$
Our idea is that $\Gamma$ collects those alternatives which necessarily belong to the superior half part of each vector in $S^U_2(p),$ because the majority relations implied by $p$ oblige them. Symmetrically, $\psi(\Gamma)$ collects those alternatives which necessarily belongs to the inferior half part of each vector in $S^U_2(p).$ To make that fact explicit we proceed by steps.

\begin{lemma}\label{catena-estesa}
Let $x,y\in N.$ If $y\in \Gamma$ and $x\hookrightarrow_{\nu} y$, then $x\in \Gamma.$
\end{lemma}

\begin{proof} Let $x, y\in N$ and $x\hookrightarrow_{\nu} y$ with $y\in \Gamma$. Thus,
there exists $z\in N$ such that $y\hookrightarrow_{\nu}z$ and $y\hookrightarrow_{\nu}\psi(z).$  By transitivity we conclude that also $x\hookrightarrow_{\nu}z$ and $x\hookrightarrow_{\nu}\psi(z)$, that is, $x\in \Gamma$.
\end{proof}

\begin{lemma}\label{psiR-lemma}The following facts hold true:\vspace{-2mm}
\begin{itemize}
\item[i)] $\Gamma\,\cap\,\psi(\Gamma)=\varnothing;$\vspace{-2mm}
\item[ii)] if $n$ is odd, then $\hat{x}_k\notin \Gamma\,\cup\, \psi(\Gamma)$.  Moreover, if $x\in \Gamma$, then $\hat{x}_k\not \hookrightarrow_{\nu} x;$\vspace{-2mm}
\item[iii)] for every $x\in N$, $|\{x,\psi(x)\}\cap \Gamma|\leq 1.$\vspace{-2mm}
\end{itemize}
\end{lemma}

\begin{proof} i) Assume that there exists $x\in N$ with $x\in \Gamma$ and $x\in \psi(\Gamma).$ Since $|\psi|=2$, this gives $\psi(x)\in \Gamma,$ so that
there exist $z,y\in N$ with
\begin{equation}\label{x}
x\hookrightarrow_{\nu} z,\quad x\hookrightarrow_{\nu} \psi(z),\quad \psi(x)\hookrightarrow_{\nu} y,\quad \psi(x)\hookrightarrow_{\nu} \psi(y).
\end{equation}
By Lemma \ref{asymmetry} applied to the second and fourth relation in \eqref{x}, we also get
\begin{equation}\label{psix2}
z\hookrightarrow_{\nu} \psi(x),\quad y\hookrightarrow_{\nu} x.
\end{equation}
From \eqref{x} and \eqref{psix2} and by transitivity of $\hookrightarrow_{\nu}$, we deduce that $\psi(x)\hookrightarrow_{\nu}  x$ and $x \hookrightarrow_{\nu} \psi(x)$, against the asymmetry of $\Sigma^C_{\nu}(p)$ established in Lemma \ref{asymmetry}.

ii) Assume that $\hat{x}_k\in \Gamma\,\cup\,\psi(\Gamma)$. Then, by i), we have $\hat{x}_k=\psi(\hat{x}_k)\in \Gamma\,\cap\,\psi(\Gamma)=\varnothing,$ a contradiction.
Next let $x\in \Gamma$ and $\hat{x}_k\hookrightarrow_{\nu} x.$ Then, by Lemma \ref{catena-estesa}, we also have $\hat{x}_k\in \Gamma$, a contradiction.

iii) Assume there is $x\in N$ such that both $x$ and $\psi(x)$ belong to $\Gamma.$ Then $x\in \psi(\Gamma)$ and so $x\in \Gamma\,\cap\,\psi(\Gamma)=\varnothing,$ against i).
\end{proof}

\begin{lemma}\label{fisso} Let $n$ be odd and $x\in N$. Then:\vspace{-2mm}
\begin{itemize}
\item[i)]$x\hookrightarrow_{\nu} \hat{x}_k$ implies $\{x,\psi(x)\}\cap \Gamma=\{x\}$;\vspace{-2mm}
\item[ii)] $\hat{x}_k\hookrightarrow_{\nu} x$ implies $\{x,\psi(x)\}\cap \Gamma=\{\psi(x)\}.$
\end{itemize}
\end{lemma}

\begin{proof} i) From $x\hookrightarrow_{\nu} \hat{x}_k$, we get $x\in \Gamma(\hat{x}_k)\subseteq \Gamma$ and thus Lemma \ref{psiR-lemma} i) gives $\psi(x)\notin \Gamma$. It follows that $\{x,\psi(x)\}\cap \Gamma=\{x\}$.

ii)  From $\hat{x}_k\hookrightarrow_{\nu} x$, using Lemma \ref{asymmetry}, we obtain $\psi(x)\hookrightarrow_{\nu} \hat{x}_k$ and i) applies to $\psi(x)$, giving $\{x,\psi(x)\}\cap \Gamma=\{\psi(x)\}.$
\end{proof}

Given $X\subseteq N$, define
\[
C_{\nu}(p,X)=\left\{f\in \mathcal{L}(X) : f \supseteq  \Sigma_\nu(p)\cap(X\times X)\right\}.
\]
Note that $C_{\nu}(p,N)=C_{\nu}(p).$
Note also that if $Y\subseteq X \subseteq N$ and $f\in C_{\nu}(p,X),$ then the restriction of $f$ to $Y$ belongs to $C_{\nu}(p,Y).$

For every $j\in K$, consider $\{\hat{x}_j,\psi(\hat{x}_j)\}\cap \Gamma$. By Lemma \ref{psiR-lemma} iii) the order of this set cannot exceed 1. Define then the sets
\[
J=\{j\in K: |\{\hat{x}_j,\psi(\hat{x}_j)\}\cap \Gamma|=1\}, \quad J^*=\{j\in K\setminus J: |\{\hat{x}_j,\psi(\hat{x}_j)\}|= 2\}.
\]
Of course, for every $j\in K\setminus J$, $\{\hat{x}_j,\psi(\hat{x}_j)\}\cap \Gamma=\varnothing.$ Note also that, when $n$ is even, we have  $J\cup J^*=K$; if $n$ is odd, by Lemma \ref{psiR-lemma} ii), we have $J\cup J^*=K\setminus\{k\}.$
For every $j\in J$, let us call $y_j$ the unique element in the set $\{\hat{x}_j,\psi(\hat{x}_j)\}\cap \Gamma$ so that $\Gamma=\{y_j:j\in J\}.$

Consider now the subset of $N$ defined by
\[
T=\{y_j : j\in J\}\cup \,\mbox{$\bigcup_{j\in J^*}$}\{\hat{x}_j,\psi(\hat{x}_j)\}
\]
and note that $T\supseteq \Gamma.$ We
consider the relation $\Sigma_{\nu}(p)\cap (T\times T)$. That relation is acyclic, because included in the acyclic relation $\Sigma_{\nu}(p)$, and thus we have $C_{\nu}(p,T)\neq \varnothing$.
Let $f\in C_{\nu}(p,T)$ and for $j\in J^*,$ let $y_j$ be the maximum of $\{\hat{x}_j,\psi(\hat{x}_j)\}$ with respect to $f,$ so that $y_j>_f \psi(y_j)$. Define
\[
M=\{y_j:j\in J\cup J^*\}.
\]
Observe that if $n$ is even, then $M\cup \psi(M)=N$; if $n$ is odd, then $M\cup \psi(M)=N\setminus\{\hat{x}_k\}$. Moreover, we have $|M|=\lfloor \frac{n}{2}\rfloor$, $M\cap \psi(M)=\varnothing$ and $\Gamma\subseteq M \subseteq T$.
The restriction of $f\in C_{\nu}(p,T)$ to $M$ is a linear order $g\in C_{\nu}(p,M),$ say
\[
g=\big[
a_1,\ldots, a_{\lfloor \frac{n}{2}\rfloor}
\big]^T,
\]
where $M=\{a_1,\ldots, a_{\lfloor \frac{n}{2}\rfloor}\}.$
 Finally,  consider the linear order on $N$ given by
\begin{equation}\label{q-magic}
q=\left\{
\begin{array}{ll}
\big[
a_1,\ldots,a_{\lfloor \frac{n}{2}\rfloor},
\psi\big(a_{\lfloor \frac{n}{2}\rfloor}\big),\ldots,\psi(a_1)
\big]^T&\mbox{ if }n\mbox{ is even}\\
\vspace{-2mm}\\
\big[a_1,\ldots,a_{\lfloor \frac{n}{2}\rfloor},
\hat{x}_k,
\psi\big(a_{\lfloor \frac{n}{2}\rfloor}\big),\ldots,\psi(a_1)
\big]^T&\mbox{ if }n\mbox{ is odd},\\
\end{array}
\right.
\end{equation}
and prove that $q\in S^U_2(p)$.
Note that in the upper part of $q$ there are the alternatives in $M$ and in the lower one those in $\psi(M)$; in the odd case, the fixed point $\hat{x}_k$ of $\psi$ is ranked in the middle, at the position
$k=\frac{n+1}{2}$.

By construction, we have that 
 $\psi q\rho_0=q,$ which implies, due to the regularity of $U$, that $q\in S^U_1(p)$.
As a consequence, we have that, for every $x,y\in N$,
\begin{equation}\label{consequence} y>_q x \  \hbox{ if and only if } \psi(x)>_q \psi(y).
\end{equation}
Indeed, by (\ref{psiR}) and (\ref{ro}), we have
$$y>_q x\quad\Leftrightarrow \quad y>_{\psi q\rho_0} x\quad\Leftrightarrow \quad x>_{\psi q} y\quad\Leftrightarrow \quad \psi(x)>_{ q} \psi(y).$$

In order to complete the proof we need to show that $q\in C_{\nu}(p)$, that is, that  for every $x,y\in N$, $x>_\nu y$ implies $x>_q y$.
Since when $n$ is even we have $N=M\cup \psi(M)$ and when $n$ is odd we have $N=M\cup \psi(M)\cup\{\hat{x}_k\},$ we reduce to prove that, for every $x,y\in M$:\vspace{-1mm}
\begin{itemize}
\item [a)] $x>_\nu y$ implies $x>_q y$;\vspace{-1mm}
\item [b)] $x>_\nu\psi(y)$ implies $x>_q \psi(y)$;\vspace{-1mm}
\item [c)] $\psi(x)>_\nu\psi(y)$ implies $\psi(x)>_q \psi(y)$;\vspace{-1mm}
\item [d)] $\psi(x)>_\nu y$ implies $\psi(x)>_q y$,\vspace{-1mm}
\end{itemize}
and, in the odd case, showing further that, for every $x\in M$:\vspace{-1mm}
\begin{itemize}
\item [e)] $x>_\nu \hat{x}_k$ implies $x>_q \hat{x}_k$;\vspace{-1mm}
\item [f)] $\hat{x}_k>_\nu x$ implies $\hat{x}_k>_q x$;\vspace{-1mm}
\item [g)] $\psi(x)>_\nu \hat{x}_k$ implies $\psi(x)>_q \hat{x}_k$;\vspace{-1mm}
\item [h)] $\hat{x}_k>_\nu \psi(x)$ implies $\hat{x}_k>_q \psi(x).$\vspace{-1mm}
\end{itemize}
Fix then $x,y\in M$ and prove first a), b), c) and d). Note that $x\neq \psi(x)$ and $y\neq\psi(y),$ because we have observed that if $\psi$ admits the fixed point $\hat{x}_k$, then $\hat{x}_k\notin M.$\vspace{-1mm}
\begin{itemize}
\item [a)] Let $x>_\nu y$. Since $g\in C_\nu(p,M)$, we have $x>_g y$ and then also $x>_q y$.\vspace{-1mm}
\item [b)] By the construction of $q$, for every $x,y\in M,$ we have $x>_q \psi(y).$\vspace{-1mm}
\item [c)] Let $\psi(x)>_\nu\psi(y)$. Then, by Lemma \ref{nupsi}, we also have $y>_\nu x$ and by a) we get $y>_q x$, which  by (\ref{consequence}) implies $\psi(x)>_q \psi(y)$.\vspace{-1mm}
\item [d)] Let us prove that the condition $\psi(x)>_\nu y$ never realizes. Indeed, assume by contradiction $\psi(x)>_\nu y$. Then, by Lemma \ref{nupsi} we also have $\psi(y)>_\nu x$.
If $y\in \Gamma$, then, by Lemma \ref{catena-estesa},  we have $\psi(x)\in \Gamma\subseteq M$ and thus $x\in M\cap \psi(M)=\varnothing$, a contradiction.
Thus, we reduce to the case $y\not\in \Gamma$. If $x\in \Gamma$, then, again by Lemma \ref{catena-estesa}, $\psi(y)\in \Gamma\subseteq M$ and we get the  contradiction $y\in M\cap \psi(M)=\varnothing$.
 If instead $x\not\in \Gamma$, then $x,\psi(x),y,\psi(y)\in T$.
As a consequence,  $\psi(x)>_\nu y$ implies $\psi(x)>_f y$ and $\psi(y)>_\nu x$ implies $\psi(y)>_f x$. Moreover, as $y$ is the maximum of $\{y,\psi(y)\}$ with respect to $f,$ we also have $y>_f\psi(y)$. By transitivity of $f$, we then get $\psi(x)>_f x,$ which is a contradiction because $x$ is the maximum of $\{x,\psi(x)\}$ with respect to $f$.\vspace{-1mm}
\end{itemize}
Assume now that $n$ is odd. Then fix $x\in M$ and prove e), f), g) and h).\vspace{-1mm}
\begin{itemize}
\item [e)] This case is trivial, because, by the construction of $q$, for every $x\in M$, we have $x>_q \hat{x}_k.$\vspace{-1mm}
\item [f)] Let us prove that it cannot be $\hat{x}_k>_\nu x$. Assume, by contradiction, $\hat{x}_k>_\nu x$. Then, by Lemma \ref{fisso} ii) we get $\psi(x)\in \Gamma\subseteq M,$ against $M\cap \psi(M)=\varnothing.$\vspace{-1mm}
\item [g)] Let us prove that it cannot be $\psi(x)>_\nu \hat{x}_k$.  Assume, by contradiction, $\psi(x)>_\nu \hat{x}_k$. Then, by Lemma \ref{fisso} i), we get
$\psi(x)\in \Gamma\subseteq M$ against $M\cap \psi(M)=\varnothing$.\vspace{-1mm}
\item [h)] This case is trivial because, by the construction of $q$, we have $\hat{x}_k>_q\psi(x)$ for all $x\in M.$
\end{itemize}

\vspace{5mm}
\noindent {\Large{\bf References}}
\vspace{2mm}

\noindent Bubboloni, D., Gori, M., 2013. Anonymous, neutral and reversal symmetric majority rules. Working Papers-Mathematical Economics 2013-05, Universit\`a degli Studi di Firenze, Dipartimento di Scienze per l'Economia e l'Impresa.
\vspace{2mm}

\noindent Bubboloni, D., Gori, M., 2014. Anonymous and neutral majority rules. Social Choice and Welfare 43, 377-401.
\vspace{2mm}

\noindent Campbell, D.E., Kelly, J.S., 2011. Majority selection of one alternative from a binary agenda.
Economics Letters 110, 272-273.
\vspace{2mm}

\noindent Campbell, D.E., Kelly, J.S., 2013. Anonymity, monotonicity, and limited neutrality: selecting a single alternative
from a binary agenda. Economics Letters 118, 10-12.
\vspace{2mm}

\noindent Kelly, J.S., 1991. Symmetry groups. Social Choice and Welfare 8, 89-95.
\vspace{2mm}

\noindent Moulin, H., 1983. \textit{The strategy of social choice}, Advanced Textbooks in Economics, North Holland Publishing Company.
\vspace{2mm}

\noindent Perry, J., Powers, R.C., 2008. Aggregation rules that satisfy anonymity and neutrality. Economics Letters 100, 108-110.
\vspace{2mm}

\noindent Powers, R.C., 2010. Maskin monotonic aggregation rules and partial anonymity. Economics Letters 106, 12-14.
\vspace{2mm}

\noindent  Quesada, A., 2013. The majority rule with a chairman. Social Choice and Welfare 40, 679-691.
\vspace{2mm}

\noindent  Rose, J.S., 1978. \textit{A course on group theory}, Cambridge University Press, Cambridge.
\vspace{2mm}

\noindent Wielandt, H., 1964. \textit{Finite permutation groups}, Academic Press, New York.
\vspace{2mm}

\end{document}